\documentclass[10pt, a4paper, reqno, english]{amsart}

\usepackage{amssymb, mathrsfs}

\usepackage[hidelinks]{hyperref}

\theoremstyle{plain}
\newtheorem*{theorem*}{Theorem}
\newtheorem{prop}{Proposition}
\newtheorem{lemma}[prop]{Lemma}
\theoremstyle{remark}
\newtheorem*{remark*}{Remark}

\numberwithin{prop}{section}
\numberwithin{equation}{section}

\newcommand\ab{\mathbf{a}}
\newcommand\Ab{\mathbf{A}}

\newcommand\oneb{\mathbf{1}}
\newcommand\ddd{\,\mathrm{d}}
\newcommand\PP{\mathbb{P}}
\newcommand\QQ{\mathbb{Q}}
\newcommand\RR{\mathbb{R}}
\newcommand\AAA{\mathbb{A}}
\newcommand\CC{\mathbb{C}}
\newcommand\ZZ{\mathbb{Z}}
\newcommand\FF{\mathbb{F}}
\newcommand\sums[1]{\sum_{\substack{#1}}}
\newcommand\ints[1]{\int_{\substack{#1}}}
\newcommand\gammab{\boldsymbol{\gamma}}
\newcommand\afr{{\mathfrak{a}}}
\newcommand\bfr{{\mathfrak{b}}}
\newcommand\cfr{{\mathfrak{c}}}
\newcommand\dfr{{\mathfrak{d}}}
\newcommand\efr{{\mathfrak{e}}}
\newcommand\pfr{{\mathfrak{p}}}
\newcommand\Xfr{{\mathfrak{X}}}
\newcommand\Yfr{{\mathfrak{Y}}}
\newcommand\Ufr{{\mathfrak{U}}}
\newcommand\afrb{{\underline{\mathfrak{a}}}}
\newcommand\cfrb{{\underline{\mathfrak{c}}}}
\newcommand\dfrb{{\underline{\mathfrak{d}}}}
\newcommand\efrb{{\underline{\mathfrak{e}}}}
\newcommand{\Cs}{\mathscr{C}}
\newcommand{\Ps}{\mathscr{P}}
\newcommand{\Is}{\mathscr{I}}
\newcommand{\Fs}{\mathscr{F}}
\newcommand{\Gs}{\mathscr{G}}
\newcommand{\Os}{\mathcal{O}}
\newcommand{\OK}{\Os_K}
\newcommand{\IK}{\Is_K}
\newcommand\Nt{\widetilde{N}}
\newcommand\Pt{\widetilde{P}}
\newcommand\Munder{\underline{M}}
\newcommand\Mover{\overline{M}}
\newcommand\Aover{\overline{A}}
\newcommand\N{\mathfrak{N}}
\newcommand{\GG}{\mathbb{G}}
\newcommand{\GGm}{\GG_{\mathrm{m}}}
\newcommand{\GGmOK}{\GG_{\mathrm{m},\OK}}
\newcommand\bigwhere[2]{\left\{#1:\ \begin{aligned}#2\end{aligned}\right\}}
\newcommand\congr[3]{#1 \equiv #2 \pmod{#3}}
\newcommand\id{\mathrm{id}}
\newcommand\Iirr{I_{\mathrm{irr}}}
\newcommand{\Mod}[1]{\ (\mathrm{mod}\ #1)}
\newcommand\fin{\mathrm{fin}}
\newcommand\an{\mathrm{an}}
\DeclareMathOperator{\vol}{vol}
\DeclareMathOperator{\Pic}{Pic}
\DeclareMathOperator{\Spec}{Spec}
\DeclareMathOperator{\Cl}{Cl}
\DeclareMathOperator{\Eff}{Eff}
\DeclareMathOperator{\rk}{rk}

\AtBeginDocument{%
   \def\MR#1{}
}

\begin{document}

\title[Integral points of bounded height on quintic del Pezzo surfaces]
{Integral points of bounded height\\ on quintic del Pezzo surfaces over number fields}

\author{Christian Bernert}

\address{Institut f\"ur Algebra, Zahlentheorie und Diskrete Mathematik, Leibniz Universit\"at Hannover, Welfengarten 1, 30167 Hannover, Germany}

\email{bernert@math.uni-hannover.de}

\author{Ulrich Derenthal} 

\address{Institut f\"ur Algebra, Zahlentheorie und Diskrete Mathematik, Leibniz Universit\"at Hannover, Welfengarten 1, 30167 Hannover, Germany}

\email{derenthal@math.uni-hannover.de}

\date{May 15, 2025}

%11G35: NT -> Arithmetic algebraic geometry (Diophantine geometry) -> Varieties over global fields
%11D45: NT -> Diophantine equations -> Counting solutions of Diophantine equations
%14G05: AG -> Arithmetic problems in algebraic geometry; Diophantine geometry -> Rational points

\setcounter{tocdepth}{1}

\begin{abstract}
  We prove an asymptotic formula for the number of integral points of bounded log-anticanonical height on split smooth quintic del Pezzo surfaces over number fields, with respect to one of the lines as the boundary divisor.
\end{abstract}
  
\maketitle

\tableofcontents

\section{Introduction}

Manin's conjecture \cite{FMT89} predicts the asymptotic behavior of the number of rational points of bounded anticanonical height on Fano varieties. One important known case is the split smooth quintic del Pezzo surface over $\QQ$, where this conjecture was proved by de la Bret\`eche \cite{Bre02} using universal torsors. Recently, we generalized his result to arbitrary number fields \cite{BD24}.

Here, we study integral points of bounded log-anticanonical height on these surfaces, which is motivated by the framework of Chambert-Loir and Tschinkel \cite{CLT10}. Previous results counting integral points include varieties with a suitable action of an algebraic group (such as partial equivariant compactifications of vector groups \cite{CLT12} and wonderful compactifications of semisimple groups of adjoint type \cite{TBT,Chow}, via harmonic analysis) and high-dimensional complete intersections (see \cite{Birch}, \cite{Skinner97}, \cite[\S 5.4]{CLT10}, \cite[\S 6.3.5]{Santens23}, via the circle method). Using universal torsors, we have results for toric varieties \cite{Wil24,Santens23} (whose universal torsors are simply open subsets of affine space), some singular quartic del Pezzo surfaces \cite{DW24,OrtmannPaper,OrtmannThesis}, and a smooth Fano threefold \cite{WilschIMRN} (whose universal torsor is a hypersurface in affine space).

A split smooth quintic del Pezzo surface over a number field $K$ is obtained as a blow-up $\pi:X \to \PP^2_K$ in the four points
\begin{equation*}
  p_1=(1:0:0),\ p_2=(0:1:0),\ p_3=(0:0:1),\ p_4=(1:1:1).
\end{equation*}
An integral model is obtained as the blow-up $\pi:\Xfr \to \PP^2_{\OK}$ of the projective plane over the ring of integers $\OK$ in the same four points.
The ten lines on $X$ are the exceptional divisors $A_i$ above $p_i$ ($1 \le i \le 4$) and the strict transforms $A_{jk}$ of the lines through $p_j$ and $p_k$ ($1 \le j<k \le 4$). Its universal torsor is defined by five Pl\"ucker equations (see \eqref{eq:torsor_i}) in ten variables corresponding to these lines.

The integral points on $\Xfr$ can be identified with the rational points on $X$ by projectivity. However, if we choose a boundary divisor $D$ on $X$, studying integral points in its complement becomes an interesting and challenging problem. Here, it seems natural to study integral points in the case where $D$ is a single line; by symmetry (with the Weyl group $S_5$ of the root system $A_4$ acting transitively on the ten lines), we may assume $D=A_{12}$. Hence we consider the integral points in the complement $\Ufr$ of the Zariski closure of $A_{12}$ in $\Xfr$.

Since integral points are expected to accumulate near the minimal strata of the Clemens complex formed by the boundary divisor, this case should have a more interesting behavior (cf. case 1 in \cite{DW24}) than the case of two intersecting lines, where the minimal stratum is just the intersection point of these two lines.

\subsection{Heights}\label{sec:heights_intro}

See Section~\ref{sec:notation} for some standard notation. We work with the following (quite general) log-anticanonical height functions: In the six-dimensional vector space of cubic forms in $K[Y_1,Y_2,Y_3]$ vanishing in the points $p_1,\dots,p_4$, we consider the four-dimensional subspace of forms vanishing on $\pi(A_{12})$ (i.e., multiples of $Y_3$). These are of the form $Y_3 \cdot P$ where $P$ is an element of the four-dimensional space of quadratic forms vanishing in $p_3,p_4$.  Let $\Ps$ be a finite generating set of this  latter space consisting of polynomials in $\OK[Y_1,Y_2,Y_3]$ satisfying
\begin{equation} \label{eq:height_gcd_i}
    \gcd_{P \in \Ps} P(y)=\gcd(y_1,y_2)\gcd(y_1-y_2,y_1-y_3),
\end{equation}
for all $y = (y_1:y_2:y_3) \in \PP^2_K(K) \setminus \{p_1,\dots,p_4\}$ (as fractional ideals, as in \cite[(1.1)]{BD24}).

Let $V \subset X$ be the complement of the ten lines. For $y \in \pi(V)(K) \subset \PP^2_K(K)$, we define
\begin{equation*}
    H_0(y) := \prod_{v \in \Omega_K} \max_{P \in \Ps} |P(y)|_v.
\end{equation*}
For $x \in V(K)$, let
\begin{equation*}
    H(x):=H_0(\pi(x)).
\end{equation*}
We will see in Lemma~\ref{lem:heights_i} that $H$ is a log-anticanonical height function.

\begin{remark*}
    Possible choices for $\Ps$ include
    \begin{align}
        \Ps_1&:=\{Y_1(Y_2-Y_3),\ Y_2(Y_1-Y_3),\ Y_1(Y_1-Y_2),\ Y_2(Y_1-Y_2)\},\label{eq:P1}\\
        \Ps_2&:=\{Y_1(Y_1-Y_3),\ Y_2(Y_2-Y_3),\ (Y_1-Y_2)^2,\ Y_3(Y_1-Y_2)\}\nonumber,
    \end{align}
    which are bases of our subspace.
    Another natural choice (see the remark after Lemma~\ref{lem:log-anticanonical_bundle_i}) is
    \begin{equation}\label{eq:P3}
        \Ps_3:=\Ps_1 \cup \Ps_2 \cup \{(Y_1-Y_2)(Y_1-Y_3),\ (Y_1-Y_2)(Y_2-Y_3)\}.
    \end{equation}
\end{remark*}

\subsection{The main result}

We determine the asymptotic behavior of the number of integral points with respect to the boundary $D=A_{12}$ of bounded log-anticanonical height outside the lines:
\begin{equation*}
    N_{\Ufr,V,H}(B) := |\{x \in \Ufr(\OK) \cap V(K) : H(x) \le B\}|.
\end{equation*}

\begin{theorem*}
    We have
    \begin{equation*}
        N_{\Ufr,V,H}(B) = c_{X,H} B(\log B)^{4+q} + O\left(\frac{B(\log B)^{4+q}}{\log \log B}\right),
    \end{equation*}
    with
    \begin{equation*}
        c_{X,H} = \alpha(X) \frac{\rho_K^4}{|\Delta_K|} \prod_{v \in \Omega_K} \omega_{H,v}(X),
    \end{equation*}
    where $\rho_K$ as in \eqref{eq:def_rho_K} is the residue of the Dedekind zeta function $\zeta_K$ at $s=1$, and 
    \begin{align*}
        \alpha(X) &= \frac{1}{2q!}\ints{t_1,\dots,t_4 \ge 0\\2t_i+2t_j-t_3-t_4 \le 1\\} \left(\frac 1 2(1-2t_1-2t_2+t_3+t_4)\right)^q \ddd t_1\cdots \ddd t_4,\\
        \omega_{H,v}(X) &=\vol\left\{y \in K_v^2 : \max_{P \in \Ps}\{|P(y_1,y_2,0)|_v\} \le 1\right\}\cdot \begin{cases}
            2, &\text{if $v$ is real,}\\
            8, &\text{if $v$ is complex,}
        \end{cases}\\
        \omega_{H,v}(X) &= \left(1-\frac 1{\N\pfr}\right)^4\left(1+\frac 4{\N\pfr}\right),\quad \text{if $v = \pfr$ is finite}.
    \end{align*}
\end{theorem*}

In Section~\ref{sec:expected}, we check that this agrees with the expected asymptotic formula from the framework of Chambert-Loir and Tschinkel \cite[\S 2]{CLT10} and the predictions of \cite[\S 2.5]{Wil24} and \cite[Conjectures~6.1, 6.6]{Santens23}. All previous verifications of these conjectures that work over arbitrary number fields (such as \cite{CLT12,TBT,Chow,Santens23}) have been obtained using harmonic analysis, which is restricted to varieties with a suitable action of an algebraic group. Here, we use the universal torsor method over number fields, which can in principle be applied to all Fano varieties.

We expect that the techniques from \cite{BD25} allow to improve the error term to $O(B(\log B)^{3+q}\log \log B)$. For this, one has to bound the divisor functions in Proposition~\ref{prop:large_moebius_i} on average instead of pointwise, as in  \cite[Proposition~4.4]{BD25}.

\begin{remark*}
    We have $\alpha(X)=\frac{17}{576}$ for $q=0$ (i.e., $K=\QQ$ and imaginary quadratic $K$).
\end{remark*}

\subsection{Overview of the proof}

Our approach is based on de la Bret\`eche's seminal article \cite{Bre02}. We simplified and refined this in order to generalize it to arbitrary number fields in our recent paper \cite{BD24}, the setting of which we build upon throughout the present work; see also \cite{BD25} for an easier version over $\QQ$. We refer to \cite[\S 1.3]{BD24} for an overview of this further development of de la Bret\`eche's approach, and use the occasion here to stress some of the key changes when studying integral points instead of rational points.

While the parameterization of rational points via universal torsors is well-studied (see \cite{Sal98,Bre02,DF14,FP16}, for example), the situation for integral points has only recently been  explored further (see \cite{WilschIMRN,DW24,OrtmannPaper,Santens23,OrtmannThesis}). 

In Section~\ref{sec:parameterization}, we pass to the universal torsor (up to the action of the N\'eron--Severi torus on integral points, for which we construct a fundamental domain) and discuss the log-anticanonical height functions. Working over number fields, an interesting aspect is that the integrality condition also enforces a multiplicative relation in the ideal classes over which the parameterization runs.

In Section~\ref{sec:restrictions}, we introduce several modifications to the counting problem, which allow us to control the error terms in the later stages.

This includes restricting variables $a_{ij}$ to their typical sizes (Section~\ref{sec:restricting_aij}), which was one of the crucial devices introduced in \cite[Section~4]{BD24} in order to deal with arbitrary number fields. In the case of rational points, the relevant region for the variables $a_{ij}$ was hyperbolic, and we were only able to truncate relatively high up in the cusps, introducing a condition $|a_{ij}| \ll WB_{ij}$ for $W$ a slowly growing function of $B$. In the present setting, however, the integrality condition reduces the situation to a compact region, so that we are actually able to obtain the condition $|a_{ij}| \ll B_{ij}$, up to a negligible error. This simplifies some of the later stages of the proof.

In the course of the argument, it proves crucial to also control the individual valuations of the unit variable $a_{12}$ (Section~\ref{sec:restricting_ai}), accounting for the new factor $(\log B)^q$ in our total count, which arises from the number of units of bounded height, as well as a modified shape of the constant $\alpha(X)$.

In Section~\ref{sec:main_contribution}, after removing the coprimality conditions by M\"obius inversion, we estimate the main contribution by interpreting the number of $a_{ij}$ as the number of lattice points in a certain region defined by the height conditions and by our fundamental domain (Section~\ref{sec:o-minimality}). This region is defined in the o-minimal structure $\RR_{\exp}$ \cite{Wilkie96}, which allows us to apply the lattice point counting result of \cite{BW14}.

A challenging aspect (which does not appear for rational points and is much simpler for integral points over $\QQ$ or imaginary quadratic fields) is the identification of the archimedean densities in the volume of this region, which we achieve by an approximation argument (Section~\ref{sec:archimedean_densities}). This is necessary since the integral points accumulate near the minimal strata of the Clemens complex; in our case, this is the boundary divisor $A_{12}$. Hence the archimedean densities are integrals over $A_{12}$.

\subsection{Notation and conventions}\label{sec:notation}

We use the standard notation for number fields as in \cite[\S 1.5]{BD24}, in particular its regulator $R_K$, its class number $h_K$, its discriminant $\Delta_K$ and its number of roots of unity $|\mu_K|$. Let $r_1$ be the number of its real embeddings, and $r_2$ the number of pairs of complex embeddings. Then $q:=r_1+r_2-1$ is the rank of $\OK^\times$. Let
\begin{equation}\label{eq:def_rho_K}
    \rho_K:=\frac{2^{r_1}(2\pi)^{r_2}R_K h_K}{|\mu_K|\cdot|\Delta_K|^{1/2}}.
\end{equation}

When we use Vinogradov's $\ll$-notation or Landau's $O$-notation, the corresponding inequalities are meant to hold for all values in the relevant range, and the implied constants may depend only on $K$ and on the choice of $\Ps$ defining the height function. We write $X_1 \asymp X_2$ for $X_1 \ll X_2 \ll X_1$.

Volumes of subsets of $\RR^n$ or $\CC^n \cong \RR^{2n}$ are computed with respect to the usual Lebesgue measure, unless stated otherwise.

For our height bound $B$, we assume $B \ge 3$. 
For $i=1,2$, we will use parameters $T_i = \exp(c_i \log B/\log \log B)$ with $c_i>0$ sufficiently large and $c_2>8c_1$.

Unless something else is stated or clear from the context, the indices $i,j,k,l$ or any subset are pairwise distinct elements of $\{1,\dots, 4\}$; when they appear in a statement, it is meant for all possible values of these indices.

\subsection*{Acknowledgements}

We thank Florian Wilsch for helpful remarks. The second author was supported by the Deutsche Forschungsgemeinschaft (DFG) -- 512730679 (RTG2965).

\section{Passage to universal torsors}\label{sec:parameterization}

In this section, the main result is Proposition~\ref{prop:parameterization_i}, which gives a parameterization of the integral points on our del Pezzo surfaces via explicit equations as well as integrality, coprimality, and height conditions on the universal torsor. This will be proved in the subsequent two subsections.

\subsection{Parameterization}

To set things up, we first need to introduce some more notation: For a $5$-tuple $\cfrb = (\cfr_0,\dots,\cfr_4)$ of nonzero fractional ideals of $\OK$, let $\Os_i = \cfr_i$ (with $\Os_{i*} = \Os_i^{\ne 0}$) and $\Os_{jk} = \cfr_0\cfr_j^{-1}\cfr_k^{-1}$. We define the ideals $\afr_i = a_i\Os_i^{-1}$, $\afr_{jk} = a_{jk}\Os_{jk}^{-1}$ for $a_i \in \Os_i$ and $a_{jk} \in \Os_{jk}$. 

Let $\Cs_0$ be a system of representatives of $\Cl_K^4$, and define
\begin{equation}\label{eq:Cs_i}
    \Cs:=\{\cfrb = (\cfr_0,\cfr_1,\dots,\cfr_4) : (\cfr_1,\dots,\cfr_4) \in \Cs_0,\ \cfr_0=\cfr_1\cfr_2\}.
\end{equation}
For $\cfrb \in \Cs$, let $u_\cfrb := \N(\cfr_0^2\cfr_3^{-1}\cfr_4^{-1}) = \N(\cfr_1^2\cfr_2^2\cfr_3^{-1}\cfr_4^{-1})$. 

As in \cite[(3.3)]{BD24}, the $\Pic X$-grading of the Cox ring of $X$ induces an action of $\GGm^5(K)$ on the universal torsor over $X$. This restricts to an action of $U_K \times (\OK^\times)^4$ on the subset of $(K^\times)^{10}$ defined by \eqref{eq:torsor_i} below; let $\Fs$ be a fundamental domain for this restricted action. We will construct $\Fs$ explicitly in Section~\ref{sec:fund_domain}.

For $P \in \Ps$, let
\begin{equation} \label{def:ptilde}
    \Pt(a_1,\dots,a_{34}):=\frac{P(a_2a_3a_{23},a_1a_3a_{13},a_1a_2a_{12})}{a_3a_4}.
\end{equation}
We define the $v$-adic factors of the height function on the universal torsor as
\begin{equation*}
    \Nt_v(x_{1v},\dots,x_{34v}) = \max_{P \in \Ps} |\Pt(x_{1v},\dots,x_{34v})|_v.
\end{equation*}

In the following parameterization of integral points, the correspondence to orbits of points on the universal torsors is analogous to \cite[\S 4.3]{OrtmannThesis} (see Section~\ref{sec:cox_rings}), while lifting our quite general height function is more complicated than in \cite[\S 4.4]{OrtmannThesis} (see Section~\ref{sec:heights}).

\begin{prop}\label{prop:parameterization_i}
    For $\cfrb \in \Cs$, let $M_\cfrb(B)$ be the set of all
    \begin{equation*}
        (a_1,a_2,a_3,a_4,a_{12},a_{13},a_{14},a_{23},a_{24},a_{34}) \in \Os_1\times\dots\times\Os_{34}
    \end{equation*}
    satisfying the torsor equations
    \begin{equation}\label{eq:torsor_i}
        \begin{aligned}
            a_4a_{14}-a_3a_{13}+a_2a_{12}&=0,\\
            a_4a_{24}-a_3a_{23}+a_1a_{12}&=0,\\
            a_4a_{34}-a_2a_{23}+a_1a_{13}&=0,\\
            a_3a_{34}-a_2a_{24}+a_1a_{14}&=0,\\
            a_{12}a_{34}-a_{13}a_{24}+a_{23}a_{14}&=0,
        \end{aligned}  
    \end{equation}
    the height condition
    \begin{equation}\label{eq:height_i}
        \prod_{v \mid \infty} \Nt_v(a_1^{(v)},\dots,a_{34}^{(v)})\le u_\cfrb B,
    \end{equation}
    the coprimality conditions
    \begin{equation}\label{eq:coprimality_i}
        \afr_i+\afr_j=\afr_i+\afr_{jk}=\afr_{ij}+\afr_{ik}=\OK,
    \end{equation}
    the fundamental domain condition
    \begin{equation}\label{eq:fundamental_domain_i}
        (a_1,\dots,a_{34}) \in \Fs \subset (K^\times)^{10},
    \end{equation}
    and the integrality condition
    \begin{equation}\label{eq:integrality_condition}
        a_{12} \in \OK^\times.
    \end{equation}

    Then
    \begin{equation*}
        N_{\Ufr,V,H}(B) = \frac{1}{|\mu_K|} \sum_{\cfrb \in \Cs}|M_\cfrb(B)|.
    \end{equation*}
\end{prop}

\begin{proof}
    This follows directly from Lemma~\ref{lem:bijection_torsor_i} and Lemma~\ref{lem:heights_i} below.
\end{proof}

\subsection{Cox rings, universal torsors, models}\label{sec:cox_rings}

The Cox ring $R$ of $X$ over $K$ (see \cite{ADHL,DP19}) is the quotient of $K[a_1,\dots,a_{34}]$ by the ideal generated by the Pl\"ucker equations \eqref{eq:torsor_i}. It contains the \emph{irrelevant ideal}
\begin{equation*}
    \Iirr = \prod_{i,j}(a_i,a_j)\prod_{i,j,k}(a_i,a_{jk})\prod_{i,j,k}(a_{ij},a_{ik}).
\end{equation*}
Let $R_{\OK}$ be the quotient of $\OK[a_1,\dots,a_{34}]$ by the same equations \eqref{eq:torsor_i}. Let $\Yfr \subset \Spec R_{\OK}$ be the complement of the closed subset defined by $\Iirr \cap R_{\OK}$. Since $X$ is a blow-up in (the preimage of) $p_4$ of the toric del Pezzo surface of degree $6$ obtained by blowing up $\PP^2_K$ in $p_1,p_2,p_3$, we check as in \cite[Proposition~4.1]{FP16} that this defines a universal torsor $\rho: \Yfr \to \Xfr$ under $\GGmOK^5$. Its base change to $K$ is a universal torsor $\rho : Y \to X$, which has the same description as $\Yfr$ in Cox coordinates except that $\OK$ is replaced by $K$.

Recall that $V$ is the complement of the lines on $X$. Its preimage $\rho^{-1}(V)$ on $Y$ is the set where all Cox coordinates $(a_1,\dots,a_{34})$ are nonzero.

For $\cfrb \in \Cs$, we consider ${}_\cfrb\rho : {}_\cfrb\Yfr \to \Xfr$, the \emph{twist} of $\rho: \Yfr \to \Xfr$ constructed in \cite[Definition~2.6]{FP16}. Let $_\cfrb\Yfr_\Ufr = ({}_\cfrb\rho)^{-1}(\Ufr) \subset {}_\cfrb\Yfr$, which is a $\GGmOK^5$-torsor over $\Ufr$ by restriction of $_\cfrb\rho$.

\begin{lemma}\label{lem:bijection_torsor_i}
    The maps ${}_\cfrb\rho$ induce a bijection
    \begin{equation*}
        \bigsqcup_{\cfrb \in \Cs} ({}_\cfrb\Yfr_\Ufr(\OK) \cap \rho^{-1}(V)(K))/\GGmOK^5(\OK) \to \Ufr(\OK) \cap V(K).
    \end{equation*}
    Explicitly, ${}_\cfrb\Yfr_\Ufr(\OK) \cap \rho^{-1}(V)(K)$ is the set of
    \begin{equation*}
        (a_1,\dots,a_{34}) \in \Os_1 \times \dots \times \Os_{34}
    \end{equation*}
    satisfying \eqref{eq:torsor_i}, \eqref{eq:coprimality_i}, \eqref{eq:integrality_condition}, and
    \begin{equation*}
        a_1\cdots a_{34} \ne 0.
    \end{equation*}
\end{lemma}

\begin{proof}
    This is completely analogous to \cite[Propositions~2.4, 2.6]{OrtmannPaper} and \cite[Propositions~4.6, 4.8]{OrtmannThesis}.
\end{proof}

\subsection{Heights}\label{sec:heights}

Recall definition~\eqref{def:ptilde}.

\begin{lemma}\label{lem:log-anticanonical_bundle_i}
    For $P \in \Ps$, the corresponding $\Pt$ is a polynomial of degree $\omega_X(A_{12})^\vee$ in the Cox ring of $X$. For any $5$-tuple $\cfrb=(\cfr_0,\dots,\cfr_4)$ of nonzero fractional ideals and $(a_1,\dots,a_{34}) \in \Os_1 \times \dots \times \Os_{34}$ satisfying \eqref{eq:torsor_i} and \eqref{eq:coprimality_i}, the greatest common divisor of $\Pt(a_1,\dots,a_{34})$ for $P \in \Ps$ is the fractional ideal $\cfr_0^2\cfr_3^{-1}\cfr_4^{-1}$. The log-anticanonical bundle $\omega_X(A_{12})^\vee$ is base point free, and $\{\Pt : P  \in \Ps\}$ is a set of generators of its space of global sections.
\end{lemma}

\begin{proof}
    Since $P \in \Ps$ is a quadratic form and $a_ia_ja_{ij}$ is a monomial of degree $\ell_0$ in the Cox ring, $\Pt$ has degree $2\ell_0-\ell_3-\ell_4 = -K_X - A_{12}$ by construction. Using the relations \eqref{eq:torsor_i}, it is easy to check that $P(a_2a_3a_{23},a_1a_3a_{13},a_1a_2a_{12})$ is a multiple of $a_3a_4$ for every $P$ in the basis $\Ps_1$ as in \eqref{eq:P1}, hence the same is true for all $P \in \Ps$. Therefore, $\Pt$ is indeed a polynomial expression in the Cox ring.

    Using \eqref{eq:height_gcd_i} and \eqref{eq:torsor_i} in the first step, $a_ia_{ij}\OK = \Os_i\Os_{ij}\afr_i\afr_{ij}=\cfr_0\cfr_j^{-1}\afr_i\afr_{ij}$ in the third step, and \eqref{eq:coprimality_i} in the final step, we compute that
    \begin{align*}
        \gcd_{P \in \Ps} \Pt(a_1,\dots,a_{34}) &= \frac{\gcd(a_2a_3a_{23},a_1a_3a_{13})\gcd(a_3a_4a_{34},a_2a_4a_{24})}{a_3a_4}\\
        &=\gcd(a_2a_{23},a_1a_{13})\gcd(a_3a_{34},a_2a_{24})\\
        &=\cfr_0^2\cfr_3^{-1}\cfr_4^{-1}(\afr_2\afr_{23}+\afr_1\afr_{13})(\afr_3\afr_{34}+\afr_2\afr_{24}) = \cfr_0^2\cfr_3^{-1}\cfr_4^{-1}.
    \end{align*}

    In particular,
    \begin{equation*}
        \{\Pt: P \in \Ps_1\}=\{a_1a_2a_{23}a_{14},a_1a_2a_{13}a_{24},a_2a_3a_{23}a_{34},a_1a_3a_{13}a_{34}\}
    \end{equation*}
    are global sections of $\omega_X(A_{12})^\vee$. It is straightforward to check that they form a basis of this space of sections. We observe that they cannot vanish simultaneously (since by definition of $\Iirr$, at most two of the coordinates $(a_1,\dots,a_{34})$ can vanish simultaneously, namely $a_i,a_{ij}$ or $a_{ij},a_{kl}$). Therefore, $\omega_X(A_{12})^\vee$ is base point free. Since every $P \in \Ps_1$ is a linear combination of $\Ps$, every $\Pt$ for $P \in \Ps_1$ is a linear combination of $\{\Pt : P  \in \Ps\}$, hence the latter generates this space of global sections.
\end{proof}

\begin{remark*}
    We observe that $\{\Pt : P \in \Ps_3\}$ (for $\Ps_3$ as in \eqref{eq:P3}) is the set of all monomials of log-anticanonical degree $\omega_X(A_{12})^\vee$ in the Cox ring.
\end{remark*}

By Lemma~\ref{lem:log-anticanonical_bundle_i}, the sections $\{\Pt : P  \in \Ps\} = \{\Pt_0,\dots,\Pt_N\}$ induce a log-anticanonical height $H'$ on $X(K)$ defined (in Cox coordinates) as
\begin{align*}
    H'((a_1:\dots:a_{34})) &= H_{\PP^N}((\Pt_0(a_1,\dots,a_{34}):\dots:\Pt_N(a_1,\dots,a_{34})))\\
    &= \prod_{v \in \Omega_K} \max_{P \in \Ps} |\Pt(a_1,\dots,a_{34})|_v.
\end{align*}

\begin{lemma}\label{lem:heights_i}
    Let $\cfrb \in \Cs$. Let $(a_1,\dots,a_{34}) \in {}_\cfrb\Yfr(\OK) \cap \rho^{-1}(V)(K)$.
    
    Then the height condition \eqref{eq:height_i} is equivalent to $H'((a_1:\dots:a_{34})) \le B$ and to $H(\rho(a_1,\dots,a_{34})) \le B$.
\end{lemma}

\begin{proof}
     We have
     \begin{equation*}
         \prod_{v \in \Omega_f} \max_{P \in \Ps}|\Pt(a_1,\dots,a_{34})|_v = \N\Big(\gcd_{P \in \Ps}\{\Pt(a_1,\dots,a_{34})\}\Big)^{-1} = \N(\cfr_0^2\cfr_3^{-1}\cfr_4^{-1})^{-1} = u_\cfrb^{-1}
     \end{equation*}
     by Lemma~\ref{lem:log-anticanonical_bundle_i}, and the archimedean factors are equal to $\Nt_v$.

    Next, we compare these heights to $H$ defined in Section~\ref{sec:heights_intro}. 
    In Cox coordinates, the map $\pi : X \to \PP^2_K$ is given by $(a_1:\dots:a_{34}) \mapsto (a_2a_3a_{23}:a_1a_3a_{13}:a_1a_2a_{12})$. Therefore,
    \begin{align*}
        H((a_1:\dots:a_{34}))&=H_0(\pi(a_1:\dots:a_{34}))\\
        &= \prod_{v \in \Omega_K} \max_{P \in \Ps}|P(a_2a_3a_{23},a_1a_3a_{13},a_1a_2a_{12})|_v
    \end{align*}
    which agrees with the above formula for $H'(a_1:\dots:a_{34})$ by definition of $\Pt$ using $\prod_v |a_3a_4|_v=1$ by the product formula.
\end{proof}

This completes the proof of Proposition~\ref{prop:parameterization_i}.

\subsection{Dependent coordinates}\label{sec:dependent_i}

To prepare for the construction of the fundamental domain $\Fs$ and for our counting argument, we describe how seven of our coordinates $(a_1,\dots,a_{34})$ determine the remaining three.

We use the notation
\begin{equation*}
    \ab'=(a_1,\dots, a_4), \quad \ab''=(a_{13},a_{14},a_{23},a_{24},a_{34}),
\end{equation*}
and
\begin{equation*}
    \Os_*'=\Os_{1*}\times\dots\times\Os_{4*}, \quad 
    \Os'' = \Os_{13}\times\Os_{14}\times\Os_{23}\times\Os_{24}\times\Os_{34}.
\end{equation*}

    We recall \cite[Lemma~3.3]{BD24}. Let $\ab' \in \Os_*'$, and $(a_{12},a_{23},a_{34}) \in \Os_{12}\times\Os_{23}\times\Os_{34}$. If 
    \begin{equation}\label{eq:mod_a4_i}
        \congr{a_3a_{23}}{a_1a_{12}}{a_4\Os_{24}}
    \end{equation}
    and
    \begin{equation}\label{eq:mod_a1_i}
        \congr{a_4a_{34}}{a_2a_{23}}{a_1\Os_{13}}
    \end{equation}
    hold, then we obtain unique
    \begin{equation}\label{eq:dependent_aij_i}
        \begin{aligned}
            a_{13}&=\frac{a_2a_{23}-a_4a_{34}}{a_1},\\
            a_{24}&=\frac{a_3a_{23}-a_1a_{12}}{a_4},\\
            a_{14}&=\frac{a_2a_3a_{23}-a_3a_4a_{34}-a_1a_2a_{12}}{a_1a_4}
        \end{aligned}
    \end{equation}
    satisfying the torsor equations \eqref{eq:torsor_i}, with $a_{13} \in \Os_{13}$ and $a_{24} \in \Os_{24}$. If additionally $\afr_1+\afr_4=\OK$, then $a_{14} \in \Os_{14}$. If \eqref{eq:mod_a4_i} or \eqref{eq:mod_a1_i} does not hold, then \eqref{eq:torsor_i} has no solution satisfying $\ab'' \in \Os''$.

\subsection{Construction of a fundamental domain} \label{sec:fund_domain}

Now we are ready to construct the fundamental domain $\Fs$ for our unit action on the universal torsor. We do this in a way that allows us to incorporate the height condition \eqref{eq:height_i}.

We define $F(\infty), F(B), \Fs_1$ as in \cite[\S 5]{FP16}. Let $\ab' \in (K^\times)^4$ and $a_{12} \in K^\times$. We define $\Nt_v(\ab';x_{12v},x_{23v},x_{34v})$ (using \cite[(3.15)]{BD24}, which is \eqref{eq:dependent_aij_i} with $a_i=a_i^{(v)}$ and $a_{jk}=x_{jkv}$) as in \cite[\S 3.4]{BD24}. For $(x_{12v})_v \in \prod_{v \mid \infty} K_v^\times$, let $S_F(\ab',(x_{12v})_v;\infty)$ be the set
\begin{equation*}
     \bigwhere{(x_{23v},x_{34v})_v \in \prod_{v\mid\infty}(K_v^\times)^2}{&\tfrac 1 2(\log \Nt_v(\ab';x_{12v},x_{23v},x_{34v}))_v \in F(\infty)\\
     &\text{$x_{13v},x_{14v},x_{24v}$ as in \cite[(3.15)]{BD24} are $\ne 0$}}.
\end{equation*}
Using the coordinate-wise diagonal embedding $\sigma : K^2 \to \prod_{v\mid\infty} K_v^2$, we define
\begin{equation*}
    \Fs_0(\ab',a_{12}):=\{(a_{23},a_{34}) \in K^2 : \sigma(a_{23},a_{34}) \in S_F(\ab',(a_{12}^{(v)})_v;\infty)\},
\end{equation*}
which is a fundamental domain for the action of $U_K$ by scalar multiplication on the set of $(a_{23},a_{34}) \in (K^\times)^2$ such that $a_{13},a_{14},a_{24}$ as in \eqref{eq:dependent_aij_i} are also nonzero.

A fundamental domain $\Fs$ as in Proposition~\ref{prop:parameterization_i} is given by the set of all tuples $(a_1,\dots,a_{34}) \in (K^\times)^{10}$ with $\ab' \in \Fs_1^4$, $(a_{23},a_{34}) \in \Fs_0(\ab',a_{12})$, and $a_{13},a_{14},a_{24}$ as in \eqref{eq:dependent_aij_i}.

We define $S_F(\ab',(x_{12v})_v;u_\cfrb B)$ as
\begin{equation*}
    \bigwhere{(x_{23v},x_{34v})_v \in \prod_{v\mid\infty} (K_v^\times)^2\!}
  {&\tfrac{1}{2}(\log \Nt_v(\ab';x_{12v},x_{23v},x_{34v}))_v \in F((u_\cfrb B)^{\frac{1}{2d}})\\
  &\text{$x_{13v},x_{14v},x_{24v}$ as in \cite[(3.15)]{BD24} are $\ne 0$}}.
\end{equation*}
Let
\begin{equation*}
    \Fs_0(\ab',a_{12};u_\cfrb B):=\{(a_{23},a_{34}) \in K^2 : \sigma(a_{23},a_{34}) \in S_F(\ab',(a_{12}^{(v)})_v;u_\cfrb B)\}.
\end{equation*}
By construction, $(a_1,\dots,a_{34}) \in (K^\times)^{10}$ satisfies \eqref{eq:height_i} and \eqref{eq:fundamental_domain_i} if and only if $\ab' \in \Fs_1^4$, $(a_{23},a_{34}) \in \Fs_0(\ab',a_{12};u_\cfrb B)$, and $a_{13},a_{14},a_{24}$ satisfy \eqref{eq:dependent_aij_i}.

\section{Restrictions of the counting problem} \label{sec:restrictions}

In preparation of the main counting argument in Section~\ref{sec:main_contribution}, we restrict the ranges for our coordinates. This is necessary to control the error terms in the next section.

For this, first we introduce a symmetry condition (Section~\ref{sec:symmetry_i}), next we restrict the variables $a_{ij}$ to their typical sizes (Section~\ref{sec:restricting_aij}), then we confine the $a_i$ to a region defining a polytope related to $\alpha(X)$ (Section~\ref{sec:restricting_ai}), and finally we remove the symmetry condition again (Section~\ref{sec:removing_symmetry}).

\subsection{Counting units}

We will often require the following lemma, which counts the number of units of bounded height.

\begin{lemma}\label{lem:count_units}
For $T \ge 1$, we have
\[\#\{a \in \Os_K^{\times}: |a|_v \le T^{d_v/d}\}=\frac{|\mu_K|}{R_K q!} (\log T)^q+O((\log T)^{q-1}+1).\]
\end{lemma}
\begin{proof}
    This is essentially \cite[Folgerung~3]{GP80}, but without the factor $d^q$ due to our different normalization of the height.
\end{proof}

\subsection{Symmetry}\label{sec:symmetry_i}

Motivated by the argument in \cite{Bre02}, we would like to exploit the symmetries of our del Pezzo surfaces to assume, up to a factor of $3$, that $a_1a_3a_4$ is not bigger than $a_1a_{23}a_{24}$ and $a_{13}a_{14}a_{34}$. However, these three terms do not form an orbit under the following symmetry group, so that we need to consider a larger orbit of six terms, accounting for a factor of $6$ instead of $3$ and a somewhat artificial condition between $a_1$ and $a_2$, the removal of which will account for the factor of $2$ in Proposition~\ref{prop:remove_symmetry_i}.

The symmetry group $S_2 \times S_3 \subset S_5$ of order $12$ acts on the set of the $10$ lines fixing $A_{12}$, with the $S_2$-component switching indices $3$ and $4$ and the $S_3$-component permuting $a_1,a_2,a_{34}$ (and the other lines accordingly). This can be visualized nicely by drawing the $10$ lines as edges of a pentagon, so that the original $S_5$-symmetry of the $10$ lines corresponds to permutations of its five vertices.

One can use this symmetry as follows: Let 
\begin{equation*}
    S=\{\id, s_1,s_2,s_1s_2,s_2s_1,s_1s_2s_1=s_2s_1s_2\}
\end{equation*} be the $S_3$-component, with $s_i$ as in \cite[\S 3.2]{BD24}. For $s \in S$, let $\Ps^{(s)}$, $s(\cfrb)$, and $\Cs^{(s)}$ be defined as in \cite[(3.8)--(3.9)]{BD24}; we note that every $(\cfr_0,\dots,\cfr_4) \in \Cs^{(s)}$ also satisfies $\cfr_0=\cfr_1\cfr_2$ (see \eqref{eq:Cs_i}).

\begin{lemma}\label{lem:symmetry_i}
    Let $\Mover_\cfrb^{(s)}(B)$ be the set of all $(a_1,\dots,a_{34}) \in \Os_1\times\dots\times\Os_{34}$ satisfying \eqref{eq:torsor_i}, \eqref{eq:coprimality_i}, \eqref{eq:fundamental_domain_i}, \eqref{eq:integrality_condition}, the symmetry condition
    \begin{equation}\label{eq:symmetry_i}
        |N(a_1a_3a_4)| \le \min\left\{\begin{aligned}
            &|N(a_2a_3a_4)|, |N(a_1a_{23}a_{24})|, |N(a_2a_{13}a_{14})|,\\
            &|N(a_{13}a_{14}a_{34})|, |N(a_{23}a_{24}a_{34})|
        \end{aligned}\right\},
    \end{equation}
    and the height condition \eqref{eq:height_i} with $\Ps$ replaced by $\Ps^{(s)}$ in the definition of $\Nt_v$. Let $\Munder_\cfrb^{(s)}(B)$ be defined similarly, with $\le$ replaced by $<$ in \eqref{eq:symmetry_i}.

    Then
    \begin{equation*}
        \sum_{s \in S}|\Munder_{s(\cfrb)}^{(s)}(B)| \le |M_\cfrb(B)| \le \sum_{s \in S}|\Mover_{s(\cfrb)}^{(s)}(B)|.
    \end{equation*}
\end{lemma}

\begin{proof}
    This is very similar to \cite[Lemma~3.2]{BD24}, except that we work with the different set $S$ as above, and the six monomials appearing in \eqref{eq:symmetry_i} form an orbit under $S$.
\end{proof}

We will estimate $|\Mover_\cfrb^{(\id)}(B)|$ for arbitrary $\Ps$ and $\cfrb$, which will give us estimates for $|\Mover_{s(\cfrb)}^{(s)}(B)|$ as well, and it will be clear that the same estimates hold for $|\Munder_{s(\cfrb)}^{(s)}(B)|$.

\subsection{Restricting the set of $\ab''$}\label{sec:restricting_aij}

For $v \mid \infty$, the typical size of the $a_{ij}^{(v)}$ is given by
\begin{equation*}
    B_{ijv}:=\frac{(u_\cfrb B|N(a_3a_4)|)^{1/(2d)}}{a_i^{(v)}a_j^{(v)}} \in K_v.
\end{equation*}
Define
\begin{equation*}
    B_{ij} := \prod_{v \mid \infty} |B_{ijv}|_v = \frac{(u_\cfrb B|N(a_3a_4)|)^{1/2}}{|N(a_ia_j)|} \asymp \frac{(B\N(\afr_3\afr_4))^{1/2}}{\N(\afr_i\afr_j)}
\end{equation*}
(with equality for $(i,j)=(1,2)$). We consider the conditions 
\begin{equation}\label{eq:condition_B12}
    |a_{12}|_v \le B_{12}^{d_v/d} \text{ for all $v \mid \infty$}
\end{equation}
and
\begin{equation}\label{eq:condition_Bij}
    |a_{ij}|_v \ll |B_{ijv}|_v  \text{ for all $v \mid \infty$}.
\end{equation}

We can now restrict the variables $a_{ij}$ to their typical sizes $B_{ij}$, using the symmetry condition \eqref{eq:symmetry_i} from above. To this end, we define
\begin{align*}
        A=A(\ab',a_{12},\cfrb,B)&:=\{\ab'' \in \Os'' : \text{\eqref{eq:torsor_i}--\eqref{eq:integrality_condition}}\},\\
    \Aover=\Aover(\ab',a_{12},\cfrb,B)&:=\{\ab'' \in \Os'' : \text{\eqref{eq:torsor_i}--\eqref{eq:integrality_condition}}, \eqref{eq:symmetry_i}\}.
\end{align*}
By the height conditions~\eqref{eq:height_i}, these sets are empty unless $\N\afr_i \ll B$.

\begin{prop}\label{prop:condition_B12}
    We have
    \begin{equation*}
        |\Mover_\cfrb^{(\id)}(B)| = \sums{\ab' \in \Os_*'\cap \Fs_1^4} \sums{a_{12} \in \OK^\times\\\eqref{eq:condition_B12}} |\Aover(\ab',a_{12},\cfrb,B)| + O(B(\log B)^{3+q}).
    \end{equation*}
\end{prop}

\begin{proof}
We need to show that the contribution of solutions with $|a_{12}|_v >1$ for some archimedean $v$ is $O(B(\log B)^{3+q})$.

For a fixed solution $(\ab',a_{12},\ab'')$ and a valuation $v\mid\infty$, write $z_{ijv}=a_{ij}^{(v)}B_{ijv}^{-1}$. Since the elements of $\Ps$ are assumed to generate the space of quadratic forms vanishing in $p_3$ and $p_4$, the height conditions now imply that
\begin{align*}
    &|z_{13v}z_{34v}|,|z_{14v}z_{34v}|,|z_{23v}z_{34v}|, |z_{24v}z_{34v}|,\\
    &|z_{13v}z_{14v}|,|z_{13v}z_{24v}|,|z_{23v}z_{14v}|,|z_{23v}z_{24v}| \ll 1,
\end{align*}
and the first four torsor equations become $z_{ijv} \pm z_{ikv} \pm z_{ilv}=0$.

We claim that for some $W_v \gg 1$, we have
\[|z_{14v}|,|z_{24v}|,|z_{34v}| \ll \frac{1}{W_v},\quad |z_{12v}|,|z_{13v}|,|z_{23v}| \ll W_v\]
or
\[|z_{13v}|,|z_{23v}|,|z_{34v}| \ll \frac{1}{W_v},\quad |z_{12v}|,|z_{14v}|,|z_{24v}| \ll W_v.\]
Indeed, this is obvious with $W_v=1$ if all the $|z_{ijv}|_v$ are $\ll 1$. Otherwise, we choose $W_v$ to be the maximum of the $|z_{ijv}|$. We note that at least two of the variables in one of the equations need to be $\gg W_v$. In particular, one of the variables $z_{13v},z_{14v},z_{23v},z_{24v}$ needs to be $\gg W_v$. Then, however, by the height conditions, we must have $|z_{34v}|_v \ll W_v^{-1}$. Hence either $z_{13v} \asymp z_{23v} \asymp W_v$ or $z_{14v} \asymp z_{24v} \asymp W_v$, and from there it is easy to deduce that $z_{12v} \asymp W_v$ from the torsor equations. Then the remaining bounds follow from the height conditions.

Now let $V_4$ be the set of $v \mid \infty$ with the first case and $V_3$ be the set of $v \mid \infty$ with the second case. Moreover, let
\begin{equation*}
    W_3=\prod_{v \in V_3} W_v \quad \text{and} \quad W_4=\prod_{v \in V_4} W_v.
\end{equation*}

We want to discard solutions with $|z_{12v}|_v>1$ for some $v$. Note that by the argument above, in all cases we have $|z_{12v}|_v \asymp \max_{i,j} |z_{ijv}|_v \asymp W_v$. Let $W=\max_v W_v$. By a dyadic decomposition of the range of $W$ and the $W_v$, it suffices to prove that the number of solutions with fixed $W_v$ and $W$ is bounded by $\ll W^{-1/4}B(\log B)^{3+q}$. 

We first consider the case where $W_3 \ge W_4$. Note that $W_3 \ge W$. 

As in Section~\ref{sec:dependent_i}, the torsor variables $a_{13},a_{24},a_{14}$ are uniquely determined by the remaining seven variables via \eqref{eq:dependent_aij_i}, and their existence is equivalent to the congruence conditions \eqref{eq:mod_a4_i} and \eqref{eq:mod_a1_i}.
As in the proof of \cite[Proposition~4.2]{BD24}, these determine $a_{23}$ and $a_{34}$ uniquely modulo $\afr_4\Os_{23}$ and $\afr_1\Os_{34}$, respectively.

The total count can thus be estimated as 
\[\sum_{a_1,a_2,a_3,a_4} \sums{a_{12}\\\max_v |z_{12v}|_v \asymp W} \sum_{a_{23} \Mod{\afr_4\Os_{23}}} \sum_{a_{34} \Mod{\afr_1\Os_{34}}} 1,\]
with the sum restricted to $|a_{23v}|_v \ll W_v^{-1}|B_{23v}|_v$ for $v \in V_3$ and $|a_{23v}|_v \ll W_v|B_{23v}|_v$ for $v \in V_4$, and similarly for $a_{34}$.
In total, the number of choices for $a_{23}$ is thus bounded by $W_3^{-1}W_4B_{23}/|N(a_4)|+1$, while the number of choices for $a_{34}$ is bounded by $W_3^{-1}W_4^{-1}B_{34}/|N(a_1)|+1$.

However, the symmetry condition $|N(a_1a_3a_4)| \le |N(a_{13}a_{14}a_{34})|$ can be rewritten as $B_{34}/|N(a_1)| \gg W_3^{1/3}W_4^{1/3}$ and the condition $|N(a_1a_3a_4)| \le N(a_1a_{23}a_{24})|$ can be rewritten as $B_{23}/|N(a_4)| \gg 1$.
Hence the number of choices for $a_{23}$ can be bounded as $\ll B_{23}/|N(a_4)|$, and the number of choices for $a_{34}$ can be bounded as $\ll W_3^{-1/3}W_4^{-1/3}B_{34}/|N(a_1)|$.

Now the total number of solutions can be bounded as
\begin{align*}
    &\ll \frac{1}{W^{1/3}}\sum_{a_1,a_2,a_3,a_4}\frac{B_{23}}{|N(a_4)|} \cdot \frac{B_{34}}{|N(a_1)|} \sums{a_{12}\\ \max_v |z_{12v}|_v \asymp W} 1\\
&=\frac{1}{W^{1/3}}\sum_{a_1,a_2,a_3,a_4} \frac{B}{|N(a_1a_2a_3a_4)|}\sums{a_{12}\\ \max_v |z_{12v}|_v \asymp W} 1.
\end{align*}

In the second case $W_4 \ge W_3$, we obtain the same estimation using a summation over $a_1,\dots,a_4,a_{12},a_{24},a_{34}$ (together with the corresponding analogue of Section~\ref{sec:dependent_i}) instead.

We now distinguish cases, depending on whether $q=0$ or $q \ge 1$.

If $q \ge 1$, then by Lemma~\ref{lem:count_units} the inner sum is bounded by $\ll (\log(WB))^{q-1}$ and the claimed bound $W^{-1/4} B(\log B)^{3+q}$ follows by summing over the $a_i$ using \cite[(4.5)]{BD24}.

On the other hand, if $q=0$, then the inner sum is empty unless $B_{12} \asymp W^{-1}$ (and $O(1)$ in the latter case). However, after fixing $a_1,a_2,a_3$, this means that $|N(a_4)|$ is confined to a dyadic interval (where we can apply \cite[Lemma~4.1]{BD24}), thus again saving one logarithm in the total bound.
\end{proof}

A variant of this argument gives a similar bound for all coordinates $a_{ij}$:

\begin{lemma}\label{lem:condition_B12_Bij}
    If $(\ab',a_{12},\ab'') \in \Os_*'\times \OK^\times\times \Os''$ satisfies \eqref{eq:condition_B12}, then it also satisfies \eqref{eq:condition_Bij}.
\end{lemma}

\begin{proof}
    We use the notation $z_{ijv}$ from the previous proof so that we need to prove that $|z_{ijv}| \ll 1$. Since $|z_{12v}| \le 1$, we have $|z_{13v}|=|z_{14v}|+O(1)$, hence the height condition $|z_{13v}z_{14v}| \ll 1$ implies that $|z_{13v}|,|z_{14v}| \ll 1$. Similarly, we obtain that $|z_{23v}|,|z_{24v}| \ll 1$. Finally, $|z_{34v}| \ll 1$ follows from the torsor equations.
\end{proof}

Combining the symmetry condition with our restriction on the coordinates $a_{ij}$, we derive the following bounds for the coordinates $a_i$, which we will use in the subsequent lemma to obtain an upper bound for $\Aover(\ab',a_{12},\cfrb,B)$.

\begin{lemma}\label{lem:estimate_Abar_i}
    Let $\ab' \in \Os_*',a_{12} \in \OK^\times$ be such that $\Aover(\ab',a_{12},\cfrb,B)$ is nonempty, and assume that \eqref{eq:condition_B12} holds. Then
    \begin{equation}\label{eq:bound_ai}
        |N(a_1^2a_3a_4)|,|N(a_2^2a_3a_4)| \ll B
    \end{equation}
    and
    \begin{equation*}
        |N(a_1)| \ll B_{34}, \quad |N(a_4)| \ll B_{23}.
    \end{equation*}
\end{lemma}

\begin{proof}
    The first inequality follows from
    \begin{equation*}
        |N(a_1a_3a_4)| \le |N(a_{13}a_{14}a_{34})| \ll B_{13}B_{14}B_{34} \ll \frac{B^{3/2}|N(a_3a_4)|^{3/2}}{|N(a_1a_3a_4)|^2},
    \end{equation*}
    which holds by symmetry and \eqref{eq:condition_Bij} via Lemma~\ref{lem:condition_B12_Bij}. Similarly, the estimate
    \begin{equation*}
        |N(a_1a_3a_4)| \le |N(a_1a_{23}a_{24})| \ll |N(a_1)|B_{23}B_{24}
    \end{equation*}
    gives the second desired bound $|N(a_2^2a_3a_4)| \ll B$.

    The final inequalities are clearly equivalent to the previous ones.
\end{proof}

\begin{lemma}\label{lem:upper_bound_Abar_i}
    For $\ab' \in \Os_*',a_{12} \in \OK^\times$ satisfying \eqref{eq:condition_B12}, we have
    \begin{equation*}
        |\Aover(\ab',a_{12},\cfrb,B)| \ll \frac{B}{|N(a_1a_2a_3a_4)|}.
    \end{equation*}
\end{lemma}

\begin{proof}
    Using Lemma~\ref{lem:condition_B12_Bij}, we estimate (as in Proposition~\ref{prop:condition_B12}) the number of solutions as
    \begin{align*}\sums{a_{23} \Mod{\afr_4\Os_{23}}\\|a_{23}|_v \ll |B_{23v}|_v}\sums{a_{34} \Mod{\afr_1\Os_{34}}\\|a_{34}|_v \ll |B_{34v}|_v} 1 &\ll \left(\frac{B_{23}}{|N(a_4)|}+1\right)\left(\frac{B_{34}}{|N(a_1)|}+1\right)
    \\&\ll \frac{B_{23}B_{34}}{|N(a_1a_4)|}\ll \frac{B}{|N(a_1a_2a_3a_4)|},
    \end{align*}
    where we use the last two inequalities in Lemma~\ref{lem:estimate_Abar_i} to remove the $+1$.
\end{proof}

\begin{remark*}
    Note that this already proves the upper bound $ N_{\Ufr,V,H}(B) \ll B(\log B)^{4+q}$ for the total count.
\end{remark*}

\subsection{Restricting the set of $\ab'$}\label{sec:restricting_ai}

Next, we must improve upon the restrictions on the coordinates $a_i$ coming from \eqref{eq:condition_B12} and \eqref{eq:bound_ai} by a relatively small but crucial factor. To this end, recall that $T_2=B^{c_2/\log\log B}$. We consider the conditions
\begin{equation}\label{eq:condition_T2_i}
    \N(\afr_1^2\afr_2^2\afr_3^{-1}\afr_4^{-1}),\N(\afr_1^2\afr_3\afr_4),\N(\afr_2^2\afr_3\afr_4) \le \frac{B}{T_2^d}
\end{equation}
(which imply $\N(\afr_i) \ll B$)
and
\begin{equation}\label{eq:condition_T2_a12_i}
    |a_{12}|_v \le \left(\frac{B_{12}}{T_2^{d/2}}\right)^{d_v/d}
\end{equation}
(which is stronger than \eqref{eq:condition_B12}).

\begin{lemma}
    We have
    \begin{equation*}
        |\Mover_\cfrb^{(\id)}(B)| = \sums{\ab' \in \Os_*' \cap \Fs_1^4\\\eqref{eq:condition_T2_i}} \sums{a_{12} \in \OK^\times\\\eqref{eq:condition_T2_a12_i}}|\Aover(\ab',a_{12},\cfrb,B)| + O\left(\frac{B(\log B)^{4+q}}{\log\log B}\right).
    \end{equation*}
\end{lemma}

\begin{proof}
    This is similar to the proof of \cite[Lemma~4.6]{BD24}. Regarding condition \eqref{eq:condition_T2_i}, it suffices to estimate the number of solutions where one of the quantities $\N(\afr_1^2\afr_2^2\afr_3^{-1}\afr_4^{-1}),\N(\afr_1^2\afr_3\afr_4),\N(\afr_2^2\afr_3\afr_4)$ is between $B/T_2^d$ and $B$.
    This interval can be partitioned into $\ll \log T_2 \ll \frac{\log B}{\log\log B}$ dyadic intervals.
    Fixing $a_1,a_2,a_3$, on each of these intervals, each of the three possible conditions restricts $|N(a_4)|$ to a certain dyadic interval. However, summing the bound from Lemma \ref{lem:upper_bound_Abar_i} yields a contribution of $O(B(\log B)^3)$ of each of those intervals. Summing this over $a_{12}$ with the help of Lemma~\ref{lem:count_units} and over the $\ll \frac{\log B}{\log\log B}$ dyadic intervals leads to the desired bound for the error term.

    Regarding condition \eqref{eq:condition_T2_a12_i}, it suffices to estimate the number of solutions where $\max_v |a_{12}|_v$ is between $(B_{12}/T_2^{d/2})^{d_v/d}$ and $B_{12}^{d_v/d}$. Again, this range can be partitioned into $\ll \frac{\log B}{\log\log B}$ dyadic intervals, and  as in the proof of Proposition~\ref{prop:condition_B12}, each of them contributes $O(B(\log B)^{3+q})$.
\end{proof}

\subsection{Removing the symmetry conditions}\label{sec:removing_symmetry}

The restrictions \eqref{eq:condition_T2_i} and \eqref{eq:condition_T2_a12_i} allow us to remove the symmetry condition \eqref{eq:symmetry_i} (replacing the set $\Aover$ by $A$).

\begin{prop}\label{prop:remove_symmetry_i}
    We have
    \begin{equation*}
        |\Mover_\cfrb^{(\id)}(B)| = \frac 1 2 \sums{\ab' \in \Os_*' \cap \Fs_1^4\\\eqref{eq:condition_T2_i}} \sums{a_{12} \in \OK^\times\\\eqref{eq:condition_T2_a12_i}}|A(\ab',a_{12},\cfrb,B)| + O\left(\frac{B(\log B)^{4+q}}{\log\log B}\right).
    \end{equation*}
\end{prop}

\begin{proof}
    First we show that we can relax \eqref{eq:symmetry_i} to $|N(a_1)|\le |N(a_2)|$ with a negligible error term. If a point with $|N(a_1)|\le |N(a_2)|$ does not satisfy \eqref{eq:symmetry_i}, then $|N(a_1a_3a_4)|$ is larger than at least one of 
    \begin{equation*}
        |N(a_1a_{23}a_{24})|,\ |N(a_2a_{13}a_{14})|,\ |N(a_{13}a_{14}a_{34})|,\ |N(a_{23}a_{24}a_{34})|
    \end{equation*}
    (since $|N(a_1a_3a_4)|\le |N(a_2a_3a_4)|$ clearly holds).

    If $|N(a_1a_3a_4)| > |N(a_1a_{23}a_{24})|$, dividing by $|N(a_1)B_{23} B_{24}| \asymp |N(a_1)|B/|N(a_2)|^2$ gives
    \begin{equation*}
        \frac{|N(a_{23}a_{24})|}{B_{23}B_{24}} \ll \frac{|N(a_2^2a_3a_4)|}{B} \ll \frac{1}{T_2^d},
    \end{equation*}
    using \eqref{eq:condition_T2_i}. If $|N(a_1a_3a_4)| > |N(a_2a_{13}a_{14})|$, then dividing by $|N(a_2)B_{13}B_{14}| \asymp |N(a_2)|B/|N(a_1)|^2$ gives
    \begin{equation*}
        \frac{|N(a_{13}a_{14})|}{B_{13}B_{14}} \ll \frac{|N(a_1^3a_3a_4)|}{|N(a_2)|B} \ll \frac{|N(a_1)|}{|N(a_2)|T_2^d} \le \frac{1}{T_2^d}.
    \end{equation*}
    Similarly, if $|N(a_1a_3a_4)| > |N(a_{13}a_{14}a_{34})|$, then
    \begin{equation*}
        \frac{|N(a_{13}a_{14}a_{34})|}{B_{13}B_{14}B_{34}} \ll \frac{|N(a_1^3a_3^{3/2}a_4^{3/2})|}{B^{3/2}} \ll \frac{1}{T_2^{3d/2}},
    \end{equation*}
    and if $|N(a_1a_3a_4)| > |N(a_{23}a_{24}a_{34})|$, then
    \begin{equation*}
        \frac{|N(a_{23}a_{24}a_{34})|}{B_{23}B_{24}B_{34}} \ll \frac{1}{T_2^{3d/2}}.
    \end{equation*}
    In any case, at least one of $|N(a_{ij})| \ll |B_{ij}|/T_2^{d/2}$ holds.

    If $(i,j)=(3,4)$, we argue as in Lemma~\ref{lem:upper_bound_Abar_i}, but bounding the summation over $a_{34}$ modulo $a_1$ by $B_{34}/(T_2^{d/2}|N(a_1)|)+1$. Here, we can leave out the $+1$ because of \eqref{eq:condition_T2_i}. In total, we obtain $\ll B/(T_2^{d/2}|N(a_1a_2a_3a_4)|)$, which is sufficient.

    If $(i,j)=(2,3)$, we argue similarly. In the other cases $(i,j)=(1,3),(1,4),(2,4)$, we choose an appropriate order of summation. 

    Finally, removing the condition $|N(a_1)| \le |N(a_2)|$ gives a factor $2$ by symmetry. Here, the case $|N(a_1)|=|N(a_2)|$ gives a negligible contribution. Indeed, $|A(\ab',a_{12},\cfrb,B)| \ll B/|a_1a_2a_3a_4|$ holds for $\ab'$ with \eqref{eq:condition_T2_i} as in Lemma~\ref{lem:upper_bound_Abar_i} since the inequalities \eqref{eq:condition_B12} and those from Lemma~\ref{lem:estimate_Abar_i} are replaced by the stronger condition \eqref{eq:condition_T2_i}, and summing over $|N(a_1)|=|N(a_2)|$ saves two factors of $\log B$.
\end{proof}

\section{Estimation of the main contribution}\label{sec:main_contribution}

\subsection{M\"obius inversion}

In the previous section, we have restricted our counting problem, resulting in Proposition~\ref{prop:remove_symmetry_i}. In order to attack this (replacing sums by integrals), we must first remove the coprimality conditions by M\"obius inversion. 

For $\dfrb=(\dfr_3,\dfr_4), \efrb=(\efr_1,\efr_2) \in \IK^2$, we consider the conditions
\begin{align}
    &\dfr_i+\afr_j=\OK,\ \dfr_3+\dfr_4=\OK,\label{eq:coprimality_d_i}\\
    &\efr_i \mid \afr_i,\label{eq:coprimality_e_i}
\end{align}
and define $\mu_K(\dfrb,\efrb):=\mu_K(\dfr_3)\mu_K(\dfr_4)\mu_K(\efr_1)\mu_K(\efr_2)$ and
\begin{align*}
    \bfr_{23} &:= (\dfr_3\cap\dfr_4)\efr_1\Os_{23},\\
    \bfr_{34} &:= (\dfr_3\cap\dfr_4)\efr_2\Os_{34}.
\end{align*}
Let $\Gs=\Gs(\ab',a_{12},\cfrb,\dfrb,\efrb)$ be the subset of $K^2$ consisting of all $(a_{23},a_{34})$ with
\begin{equation*}
    a_{23} \in a_{12}\gamma_{23}+\afr_4\bfr_{23},\quad a_{34} \in a_{23}\gamma_{34}+\afr_1\bfr_{34},
\end{equation*}
where $\gamma_{23}=\gamma_{23}^*/a_3, \gamma_{34} = \gamma_{34}^*/a_4 \in K$ for $\gamma_{23}^*,\gamma_{34}^* \in \OK$ satisfying
\begin{align*}
    &\congr{\gamma_{23}^*}{0}{\afr_3\cfr_1\dfr_3\efr_1},\quad \congr{\gamma_{23}^*}{a_1}{\afr_4\cfr_1\dfr_4},\\
    &\congr{\gamma_{34}^*}{0}{\afr_4\cfr_2\dfr_4\efr_2},\quad \congr{\gamma_{34}^*}{a_2}{\afr_1\cfr_2},
\end{align*}
as in \cite[\S 5.1]{BD24}. Defining $\gammab = \gammab(\ab',a_{12},\cfrb,\dfrb,\efrb) = (a_{12}\gamma_{23},a_{12}\gamma_{23}\gamma_{34}) \in K^2$, we observe that
\begin{equation*}
    \Gs = \gammab+\Gs',
\end{equation*}
where $\Gs'=\Gs'(\ab',\cfrb,\dfrb,\efrb)$ is the additive subgroup of $K^2$ consisting of all $(a_{23},a_{34})$ with
\begin{equation*}
    a_{23} \in \afr_4\bfr_{23},\quad a_{34} \in a_{23}\gamma_{34}+\afr_1\bfr_{34}.
\end{equation*}
Let $\afrb'=(\afr_1,\dots,\afr_4)$, and
\begin{equation*}
    \theta_0(\afrb'):=\begin{cases}
        1, & \text{$\afr_i+\afr_j=\OK$,}\\
        0, & \text{else.}
    \end{cases}
\end{equation*}

With this setup, we can now rewrite our counting function in terms of a lattice point counting problem.

\begin{prop}\label{prop:moebius_i}
    Let $\ab' \in \Os_*'$ with $\theta_0(\afrb')=1$, and $a_{12} \in \OK^\times$. Then
    \begin{equation*}
        |A(\ab',a_{12},\cfrb,B)| = \sums{\dfrb:\eqref{eq:coprimality_d_i}\\\efrb:\eqref{eq:coprimality_e_i}} \mu_K(\dfrb,\efrb)|\Gs(\ab',a_{12},\cfrb,\dfrb,\efrb)\cap \Fs_0(\ab',a_{12};u_\cfrb B)|.
    \end{equation*}
\end{prop}

\begin{proof}
    This is the special case $a_{12} \in \OK^\times$ of \cite[Proposition~5.2]{BD24}. Here,
    \begin{equation*}
        \bfr_{12}:=(\dfr_1\cap\dfr_2\cap(\dfr_3+\dfr_4)\efr_3\efr_4\Os_{12}
    \end{equation*}
    divides $\afr_{12} = \Os_{12} = \OK$, hence  $\dfr_1=\dfr_2=\dfr_3+\dfr_4=\efr_3=\efr_4=\OK$.
\end{proof}

\subsection{The case of large M\"obius variables}

As in \cite[Section 5.2]{BD24}, in order to control the error terms, it is crucial to truncate the range of the M\"obius variables to a suitably small size $T_1=\exp(c_1 \log B/\log\log B)$. This is achieved in Proposition~\ref{prop:large_moebius_i} using the following preliminary upper bound for the total number of solutions without the coprimality conditions, combined with a lifting argument.

\begin{lemma}\label{lem:upper_bound}
    We have
    \begin{equation*}
        |\{(\ab',a_{12},\ab'') \in \Os_*' \times \OK^\times \times \Os'': \eqref{eq:torsor_i}, \eqref{eq:height_i}, \eqref{eq:fundamental_domain_i}, \eqref{eq:integrality_condition}\}| \ll B(\log B)^{6+q}.
    \end{equation*}
\end{lemma}

\begin{proof}
    As in Lemma \ref{lem:symmetry_i}, it suffices to bound the number of solutions satisfying the symmetry condition \eqref{eq:symmetry_i}. Following the proof of Proposition \ref{prop:condition_B12}, we can bound the number of solutions where \eqref{eq:condition_B12} is violated by
    \[(\log B)^q \sum_{\N\afr_i \ll B} \frac{B}{\N(\afr_1\afr_2\afr_3\afr_4)} \N((\afr_1+\afr_4)(\afr_3+\afr_4))\]
    which is $\ll B(\log B)^{6+q}$ as in the proof of \cite[Lemma~5.3]{BD24}.
    
    On the other hand, when \eqref{eq:condition_B12} holds, we can repeat the argument from the proof of Lemma \ref{lem:upper_bound_Abar_i} and obtain the same bound.
\end{proof}

\begin{prop}\label{prop:large_moebius_i}
    We have
    \begin{align*}
        |\Mover_{\cfrb}^{(\id)}(B)| ={}& \frac 1 2 \sums{\ab' \in \Os_*' \cap \Fs_1^4\\\eqref{eq:condition_T2_i}} \theta_0(\afrb') \sums{a_{12} \in \OK^\times\\\eqref{eq:condition_T2_a12_i}} \sums{\dfrb: \eqref{eq:coprimality_d_i},\ 
        \N\dfr_i \le T_1\\
        \efrb: \eqref{eq:coprimality_e_i},\ \N\efr_i \le T_1} \mu_K(\dfrb,\efrb)\cdot \\
        &\qquad |\Gs(\ab',a_{12},\cfrb,\dfrb,\efrb)\cap \Fs_0(\ab',a_{12};u_\cfrb B)|+O\left(\frac{B(\log B)^4}{\log\log B}\right).
    \end{align*}
\end{prop}

\begin{proof}
    By Proposition~\ref{prop:remove_symmetry_i} and Proposition~\ref{prop:moebius_i}, it suffices to discard all tuples $(\ab',\ab'',\dfrb,\efrb)$ with $\max(\N\dfr_3,\N\dfr_4,\N\efr_1,\N\efr_2) >T_1$. As in the proof of \cite[Proposition~5.4]{BD24}, by passing first to coprime and then to principal ideals, we can reduce to the case where $\dfrb, \efrb$ are replaced by principal ideals $(d_i),(e_i)$ such that $\max \{|N(d_3d_4)|,|N(e_1)|,|N(e_2)|\} \gg T_1$.
    
    Acting with
    \[\left(\frac{1}{e_1e_2},\frac{1}{e_1},\frac{1}{e_2}, d_3,d_4\right) \in (K^{\times})^5\]
    on the tuples $(\ab',\ab'')$ gives a new solution
    \[\left(\frac{a_1}{e_1},\frac{a_2}{e_2},a_3d_3,a_4d_4,\frac{a_{13}}{d_4e_2},\frac{a_{14}}{d_4e_2},\frac{a_{23}}{d_4e_1},\frac{a_{24}}{d_4e_1},\frac{a_{34}}{d_3d_4e_1e_2}\right)\]
    of height $\ll B/|N(d_3d_4e_1^2e_2^2)|$. By Lemma \ref{lem:upper_bound}, the number of such solutions is bounded by
    \[\frac{B(\log B)^{6+q}}{|N(d_3d_4e_1^2e_2^2)|}.\]
    As we can reconstruct the $d_i$ from the new solution up to a divisor function, the argument from the proof of \cite[Proposition~5.4]{BD24} now shows that the number of solutions with $\max \{|N(d_3d_4)|,|N(e_1)|,|N(e_2)|\} \gg T_1$ can be estimated as $\ll T_1^{-1/2}B(\log B)^6$, which is satisfactory.
\end{proof}

\subsection{Counting via o-minimal structures}\label{sec:o-minimality}

In order to estimate the number of points in the shifted lattice $\Gs$ lying in the set $\Fs_0$ defined by our fundamental domain and the height conditions, the key observation is that $\Fs_0$ is defined in an o-minimal structure so that the counting result from \cite{BW14} can be applied (as in \cite[\S 8--10]{FP16}).

\begin{lemma}\label{lem:o-minimal_counting_i}
    For $\ab' \in \Os_*'$ and $a_{12} \in \OK^\times$ satisfying \eqref{eq:condition_T2_a12_i}, we have
    \begin{align*}
        |\Gs(\ab',a_{12},\cfrb,\dfrb,\efrb) \cap \Fs_0(\ab',a_{12};u_\cfrb B)| =&{}  \frac{2^{2r_2}\vol S_F(\ab',(a_{12}^{(v)})_v; u_\cfrb B)}{|\Delta_K|\N(\afr_1\afr_4\bfr_{23}\bfr_{34})}\\ &+ O\left(\frac{B}{T_2^{1/2}|N(a_1\cdots a_4)|}\right).
    \end{align*}
\end{lemma}

\begin{proof}
    Let
    \begin{equation*}
        \tau = (\tau_v)_v: \prod_{v\mid\infty} K_v^2 \to \prod_{v\mid\infty} K_v^2, \quad (x_{23v},x_{34v})_v \mapsto \left(\frac{x_{23v}}{\N(\afr_4\bfr_{23})^{1/d}}, \frac{x_{34v}}{\N(\afr_1\bfr_{34})^{1/d}}\right)_v.
    \end{equation*}
    Defining
    \begin{equation*}
        \Fs_0' = \Fs_0'(\ab',a_{12};u_\cfrb B) := \Fs_0-\gammab,\quad S_F'=S_F'(\ab',(a_{12}^{(v)})_v;u_\cfrb B):=S_F-\sigma(\gammab), 
    \end{equation*}
    and $\Lambda := \tau(\sigma(\Gs'))$, we have
    \begin{equation*}
        |\Gs \cap \Fs_0| = |\Gs' \cap \Fs_0'| = |\sigma(\Gs') \cap S_F'| = |\Lambda \cap \tau(S_F')|.
    \end{equation*}
    Here, by our choice of $\tau$, we note that $\Lambda$ is a lattice of rank $2$ with determinant $1$ and first successive minimum $\lambda_1 \ge 1$. Furthermore, $\tau(S_F')$ is the fiber of a set that is definable in the o-minimal structure $\RR_{\exp}$ (as in \cite[Lemma~5.5]{BD24}, without the condition involving $\beta'=W$, but with the additional parameters $(\beta_{23v}',\beta_{34v}')_v=\sigma(\gammab) \in \prod_{v \mid \infty} K_v^2$), so that we can apply \cite[Theorem~1.3]{BW14}.
    
    As in \cite[Lemma~5.6]{BD24}, up to a bounded constant, the error term is bounded by the sum of the volumes $V_S$ of the nontrivial orthogonal projections of $\tau(S_F')$, which are the same as for $\tau(S_F)$. By \eqref{eq:condition_T2_a12_i} and as in Lemma~\ref{lem:condition_B12_Bij}, $\tau(S_F) \subset \prod_{v \mid \infty} \tau_v(S_{F,v})$, where
    \begin{equation*}
        S_{F,v}:=\{(x_{23v},x_{34v}) \in K_v^2 : \text{$|x_{ijv}|_v \le c_v|B_{ijv}|_v$ for all $i,j$}\}
    \end{equation*} with sufficiently large $c_v>0$. Here, the volume of the orthogonal projection of $\tau_v(S_{F,v})$ corresponding to $P_v=(p_{23v},p_{34v}) \in \{0,\dots,d_v\}^2 \setminus \{(0,0)\}$ (see \cite[Lemma~5.6]{BD24}) is
    \begin{equation*}
        V_{P_v} \ll \left(\frac{B}{|N(a_1\cdots a_4)|}\right)^{d_v/d} \left(\frac{B_{23}}{|N(a_4)|}\right)^{-p_{23v}/d} \left(\frac{B_{34}}{|N(a_1)|}\right)^{-p_{34v}/d}.
    \end{equation*}
    Using $V_S \le \prod_{v\mid\infty}V_{P_v}$ for certain $P_v$, which are not all $(0,0)$, we conclude
    \begin{equation*}
        V_S \ll \frac{B}{|N(a_1\cdots a_4)|}\left(\frac{B_{23}}{|N(a_4)|}\right)^{-e_{23}}\left(\frac{B_{34}}{|N(a_1)|}\right)^{-e_{34}},
    \end{equation*}
    where $e_{23},e_{34}$ are nonnegative and at least one is $1/d$ or larger.
    By \eqref{eq:condition_T2_i},
    \begin{equation*}
        \left(\frac{B_{23}}{|N(a_4)|}\right)^{-1/d},\left(\frac{B_{34}}{|N(a_1)|}\right)^{-1/d} \ll T_2^{-1/2}.
    \end{equation*}
    Therefore, $V_S \ll T_2^{-1/2} B/|N(a_1\cdots a_4)|$.
\end{proof}

Next, we perform the summation over the variables $\dfr_i,\efr_i$ coming from the M\"obius inversions. In the main term, this results in the arithmetic function
\begin{equation}\label{eq:def_theta}
    \theta(\afrb',T_1):=\sums{\dfrb: \eqref{eq:coprimality_d_i},\ \N\dfr_i \le T_1\\
        \efrb: \eqref{eq:coprimality_e_i},\ \N\efr_i \le T_1} \frac{\mu_K(\dfrb,\efrb)}{\N(\bfr_{23}\Os_{23}^{-1}\bfr_{34}\Os_{34}^{-1})}.
\end{equation}
To control the summation of the error terms from Lemma~\ref{lem:o-minimal_counting_i}, our truncation of the M\"obius variables in Proposition~\ref{prop:large_moebius_i} is crucial.

\begin{prop}\label{prop:counting_result_i}
    We have
    \begin{align*}
        |\Mover_\cfrb^{(\id)}(B)| ={}& \frac 1 2 \sums{\ab' \in \Os_*' \cap \Fs_1^4\\\eqref{eq:condition_T2_i}} \sums{a_{12} \in \OK^\times\\\eqref{eq:condition_T2_a12_i}} \frac{2^{2r_2}\theta_0(\afrb')\theta(\afrb',T_1)\vol S_F(\ab',(a_{12}^{(v)})_v;u_\cfrb B)}{|\Delta_K|\N(\afr_1\afr_4\Os_{23}\Os_{34})}\\ &+ O\left(\frac{B(\log B)^{4+q}}{\log\log B}\right).
    \end{align*}
\end{prop}

\begin{proof}
    See \cite[Proposition~5.7]{BD24}. We obtain the main term directly from Proposition~\ref{prop:large_moebius_i} combined with Lemma~\ref{lem:o-minimal_counting_i}. For the error term, we use Lemma~\ref{lem:count_units} for the summation over $a_{12}$ and \cite[(4.5)]{BD24} for the summation over $a_1,\dots,a_4$; the summation over $\N\dfr_i,\N\efr_i\le T_1$ gives a factor $T_1^4$, which is sufficient since $c_2 > 8c_1$.
\end{proof}

\subsection{Archimedean densities}\label{sec:archimedean_densities}

Now we compute the volume of the set $S_F$ appearing in the main term. At first glance, the outcome has a similar shape as for rational points (see \cite[Lemma~5.1]{FP16} and \cite[\S 5.4]{BD24}). However, for integral points, the archimedean density is an integral over the boundary divisor, which is obtained by replacing the corresponding (small) variable $a_{12}$ by $0$, resulting in an error term. Our treatment of this step here seems to be much more conceptual than the ad-hoc computations in \cite[\S 5]{DW24}, for example.

We will need the following lemma (which might be well-known).

\begin{lemma}\label{lem:delta}
    Let $M \subset \RR^n$ be $(n-1)$-Lipschitz parameterizable. Let $\delta>0$. Then the $\delta$-neighborhood of $M$ has volume $\ll_M \delta$.
\end{lemma}

\begin{proof}
    Let $I = [0,1]^{n-1}$ be the $(n-1)$-dimensional unit cube. By assumption, $M$ is the union of the images of finitely many Lipschitz maps $\phi_i: I \to M$ with Lipschitz constant $C_i>0$ (i.e., satisfying $|\phi_i(x)-\phi_i(y)| \le C_i|x-y|$ for all $x,y \in I$).

    We may assume $\delta=1/N$ for some integer $N$. We cut $I$ into $N^{n-1}$ cubes $I_j$ of side length $1/N$. Then $\phi_i(I_j)$ has diameter at most $C_i/N$. Therefore, the $\delta$-neighborhood of $\phi_i(I_j)$ has diameter at most $(C_i+2)/N$ and hence volume $\ll_{C_i} 1/N^n$. Since the $\delta$-neighborhood of $\phi_i(I)$ is the union of the $\delta$-neighborhoods of $\phi_i(I_j)$, its volume is $\ll_{C_i} N^{n-1}/N^n = 1/N = \delta$. Finally, we observe that the $\delta$-neighborhood of $M$ is the union of the $\delta$-neighborhoods $\phi_i(I)$ for our finitely many Lipschitz maps $\phi_i$.
\end{proof}

\begin{prop}\label{prop:real_density_i}
    For $\ab' \in \Os_*' \cap \Fs_1^4$ with \eqref{eq:condition_T2_i} and $a_{12} \in \OK^\times$ with \eqref{eq:condition_T2_a12_i}, we have
    \begin{equation*}
        \vol(S_F(\ab',(a_{12}^{(v)})_v;u_\cfrb B)) = \frac{R_K u_\cfrb B}{2\cdot 2^{2r_2}|N(a_2a_3)|} \left(\prod_{v \mid \infty} \omega_{H,v}(X)+O(T_2^{-d})\right).
    \end{equation*}
\end{prop}

\begin{proof}
    We recall the definition of $S_F$ from Section~\ref{sec:fund_domain}. By introducing the substitutions $x_{23v}=B_{23v}z_{23v}$ and $x_{34v}=B_{34v}z_{34v}$ with Jacobian 
    \[\prod_{v \mid \infty} B_{23v}B_{34v}=\frac{u_\cfr B}{|N(a_3a_4)|},\]
    as well as the notation $a_{12}^{(v)}=B_{12v}z_{12v}$, we find as in \cite[Lemma~5.8]{BD24} that
    \begin{equation*}
        \vol(S_F(\ab',(a_{12}^{(v)})_v; u_\cfr B))=\frac{u_\cfr B}{|N(a_3a_4)|} \cdot \vol(S_F(\oneb,(z_{12v})_v; 1)).
    \end{equation*}
    
    Let
    \begin{equation*}
        N_v(y_1,y_2,y_3):=\max_{P \in \Ps}|P(y_1,y_2,y_3)|_v.
    \end{equation*}
    As in \cite[Lemma~5.10]{BD24} (with the coordinate-wise exponential function $\exp$), we obtain
    \begin{equation*}
        \vol(S_F(\oneb,(z_{12v})_v;1)) = \int_{(N_v(y_{1v},y_{2v},z_{12v}))_{v \mid \infty} \in \exp(2F(1))} \prod_{v \mid \infty} \ddd y_{1v} \ddd y_{2v}.
    \end{equation*}
    
    The next step is to replace $z_{12v}$ by zero. Indeed, note that $\exp(2F(1))$ and hence all the values of $N_v$ and hence all the values of $|P|$ in consideration are bounded. So in particular, $|y_{1v}-y_{2v}|$ and then also $|y_{1v}|$ and $|y_{2v}|$ are bounded by inspection of the set $\Ps_3$ as in \eqref{eq:P3} (using that the third coordinate $z_{12v}$ is bounded). Thus, by the polynomial definition of $N_v$ and using $z_{12v} \ll T_2^{-d}$ (by \eqref{eq:condition_T2_a12_i}), we obtain $N_v(y_{1v},y_{2v},z_{12v})=N_v(y_{1v},y_{2v},0)+O(T_2^{-d})$.

    We can thus write
    \begin{align*}
        &\int_{(N_v(y_{1v},y_{2v},z_{12v}))_{v} \in \exp(2F(1))} \prod_{v \mid \infty} \ddd y_{1v} \ddd y_{2v}\\
        ={}&\int_{(N_v(y_{1v},y_{2v},0))_{v} \in \exp(2F(1))} \prod_{v \mid \infty} \ddd y_{1v} \ddd y_{2v}+O\!\left(\int_{(N_v(y_{1v},y_{2v},0))_{v} \in D} \prod_{v \mid \infty} \ddd y_{1v} \ddd y_{2v}\right)\!,
    \end{align*}
    where $D$ is the $\delta$-neighborhood of the boundary of $\exp(2F(1))$, for $\delta \ll T_2^{-d}$.
    
    To compute the main term and to bound the error term, we follow the strategy in \cite[Lemma~5.1]{FP16}. Indeed, defining the Lebesgue-measurable function
    \begin{equation*}
        f:\prod_{v|\infty} K_v^2 \to \RR_{\ge 0}^{\Omega_\infty}, \quad (y_{1v},y_{2v}) \mapsto (N_v(y_{1v},y_{2v},0))_{v \mid \infty},
    \end{equation*}
    we see that the main term is just $f_*(\vol)(\exp(2F(1)))$, whereas the error term is $f_*(\vol(D))$.

    As in \cite[Lemma~5.1]{FP16}, we find that
    \begin{align*}
        f_*(\vol)&=\prod_{v | \infty} \vol\{(y_{1v},y_{2v}) \in K_v^2: N_v(y_{1v},y_{2v},0) \le 1\} \cdot \vol\\
        &=\frac{1}{2^{r_1+3r_2}} \prod_{v|\infty} \omega_{H,v}(X) \cdot \vol
    \end{align*}
    by computing both sides on elementary cells. In particular, we see using Lemma~\ref{lem:delta} that $f^*(\vol)(D) \ll \vol(D) \ll T_2^{-d}$ since the boundary of the fundamental domain $\exp(2F(1)) \subset \RR^{\Omega_\infty}$ is $(|\Omega_\infty|-1)$-Lipschitz parameterizable (see \cite[\S VI.3]{Lang}).

    Finally, as in \cite[Lemma~5.1]{FP16}, we have
    \[\vol(\exp(2F(1)))=2^{r_1+r_2-1} R_K.\qedhere\]
\end{proof}

Recall the notation from Section~\ref{sec:symmetry_i}. In the final step of the proof, the following observation will be needed to sum over $s \in S$ coming from the symmetries.

\begin{lemma}\label{lem:densities_under_symmetry_i}
For $s \in S$ and $v | \infty$, let $\omega_{H,v}^{(s)}(X)$ be defined as $\omega_{H,v}(X)$, but with $\Ps$ replaced by $\Ps^{(s)}$. Then $\omega_{H,v}^{(s)}(X)=\omega_{H,v}(X)$.
\end{lemma}

\begin{proof}
    By definition \eqref{def:ptilde}, we have
    \begin{align*}
        |P^{(s_2)}(y)|_v&=|\Pt^{(s_2)}(1,1,1,1,0,y_2,y_2,y_1,y_1,y_1-y_2)|_v\\
        &=|\Pt(y_1-y_2,1,y_2,y_2,0,1,1,y_1,y_1,1)|_v\\
        &=|P(y_1,y_1-y_2,0)|_v
    \end{align*}
and similarly $|P^{(s_1)}(y)|_v=|P(y_1-y_2,-y_2,0)|_v$. Since the Jacobian of the changes of variables $(y_1,y_2) \mapsto (y_1,y_1-y_2)$ and $(y_1,y_2) \mapsto (y_1-y_2,-y_2)$ has absolute value $1$, we conclude that the volumes in the definition of $\omega_{H,v}(X)$, $\omega_{H,v}^{(s_1)}(X)$ and $\omega_{H,v}^{(s_2)}(X)$ agree, and thus $\omega_{H,v}^{(s)}(X)$ is the same for all $s \in S$.
\end{proof}

\subsection{$\pfr$-adic densities}

In the following lemma, we remove the truncation of the M\"obius variables $\dfr_i, \efr_i$ in the definition \eqref{eq:def_theta} of $\theta(\afrb',T_1)$ and expand the outcome as an Euler product of $\pfr$-adic densities.

\begin{lemma}\label{lem:p-adic_density_i}
    For $\ab' \in \Os_*'$ with $\theta_0(\afrb')=1$, we have
    \begin{equation*}
        \theta(\afrb',T_1) = \theta(\afrb')+O(T_1^{-1/2}),
    \end{equation*}
    where
    \begin{equation*}
        \theta(\afrb') = \prod_{\pfr \nmid \afr_1\cdots \afr_4} \left(1-\frac 2{\N\pfr^2}\right) \prod_{\pfr \mid \afr_1\afr_2} \left(1-\frac 1{\N\pfr}\right) \prod_{\pfr \mid \afr_3\afr_4} \left(1-\frac 1{\N\pfr^2}\right).
    \end{equation*}
    The total contribution of the error term to $|\Mover_\cfrb^{(\id)}(B)|$ is $O(T_1^{-1/2} B(\log B)^{4+q})$, which is sufficient.
\end{lemma}

\begin{proof}
    As in \cite[Lemma~5.12]{BD24}, but without requiring Rankin's trick for the estimate of the error terms, we can write
    \[\theta(\afrb',T_1)=D(\afrb',T_1)E(\afr_1,T_1)E(\afr_2,T_1)\]
    with
    \[E(\afr,T_1)=\sum_{\efr \mid \afr,\ \N\efr \le T_1} \frac{\mu_K(\efr)}{\N\efr}=\prod_{\pfr \mid \afr} \left(1-\frac{1}{\N\pfr}\right)+O(T_1^{-1/2}),\]
    and
    \[D(\afrb',T_1)=\sum_{\dfrb: \eqref{eq:coprimality_d_i},\ \N\dfr_i \le T_1} \frac{\mu_K(\dfr_3\dfr_4)}{\N(\dfr_3^2\dfr_4^2)}=\sum_{\dfrb: \eqref{eq:coprimality_d_i}} \frac{\mu_K(\dfr_3\dfr_4)}{\N(\dfr_3^2\dfr_4^2)} +O(T_1^{-1})\]
    from where it is straightforward to compute the Euler product and the contribution of the error term as in the statement of our result.
\end{proof}

\subsection{Completion of the proof}

We collect the results obtained thus far:

\begin{prop}\label{prop:ideals_i}
    We have
    \begin{align*}
        \sum_{\cfrb \in \Cs} |\Mover_\cfrb^{(\id)}(B)| ={}& \frac{R_K(\prod_{v \mid \infty} \omega_{H,v}(X)+O(T_2^{-d}))}{2^2|\Delta_K|} \sums{\afrb' \in \IK^4\\\eqref{eq:condition_T2_i}} \frac{\theta_0(\afrb')\theta(\afrb')B}{\N(\afr_1\afr_2\afr_3\afr_4)} \sums{a_{12} \in \OK^\times\\\eqref{eq:condition_T2_a12_i}} 1\\ &+ O\left(\frac{B(\log B)^{4+q}}{\log\log B}\right).
    \end{align*}
\end{prop}

\begin{proof}
    We combine Proposition~\ref{prop:counting_result_i} with Proposition~\ref{prop:real_density_i} and Lemma~\ref{lem:p-adic_density_i}. We use $|N(a_2a_3)|\N(\Os_{23}\Os_{34}) = u_\cfrb \N(\afr_2\afr_3)$ and turn the summations over $\cfrb \in \Cs$ and $\ab' \in \Os_*' \cap \Fs_1^4$ into a summation over $\afrb' \in \IK^4$.
\end{proof}

At this point, we can perform the summation over $\afrb'$ as in \cite[\S 7]{DF14}, with the summation over $a_{12}$ controlled by Lemma~\ref{lem:count_units}. Here, we obtain the Euler product
\begin{equation*}
    \theta_1 := \prod_\pfr \left(1-\frac 1{\N\pfr}\right)^4\left(1+\frac 4{\N\pfr}\right) = \prod_\pfr \omega_{H,\pfr}(X)
\end{equation*}
as the average of $\theta_0(\afrb')\theta(\afrb')$ over $\afr_1,\dots,\afr_4$ in the sense of \cite[Lemma~2.8]{DF14}.

\begin{lemma}\label{lem:summation_over_afr_i}
    We have
    \begin{equation*}
        \sums{\afrb' \in \IK^4\\\eqref{eq:condition_T2_i}} \frac{\theta_0(\afrb')\theta(\afrb')B}{\N(\afr_1\cdots\afr_4)} \sums{a_{12} \in \OK^\times\\\eqref{eq:condition_T2_a12_i}} 1 = \frac{|\mu_K|}{R_K}\rho_K^4V_1\theta_1B(\log B)^{4+q}+O\left(\frac{B(\log B)^{4+q}}{\log \log B}\right),
    \end{equation*}
    where
    \begin{equation*}
        V_1 = \frac{1}{q!}\int_{P} \left(\frac 1 2(t_3+t_4+1)-t_1-t_2\right)^q \ddd t_1\cdots \ddd t_4,
    \end{equation*}
    with the polytope
    \begin{equation*}
        P = \bigwhere{(t_1,\dots,t_4) \in \RR_{\ge 0}^4}{
            &2t_1+t_3+t_4 \le 1,\\
            &2t_2+t_3+t_4 \le 1,\\
            &2t_i+2t_j-t_3-t_4 \le 1
        }.
    \end{equation*}
\end{lemma}

\begin{proof}
    By Lemma~\ref{lem:count_units}, the inner sum over $a_{12}$ is
    \begin{equation*}
        \frac{|\mu_K|}{R_K q!} (\log(B_{12}/T_2^{d/2}))^q +O((\log B)^{q-1}).
    \end{equation*}
    When summing its error term over $\afrb'$, we can use $|\theta_0(\afrb')\theta(\afrb')| \le 1$ and obtain a total contribution of $O(B(\log B)^{3+q})$.
    
    To sum the main term, we divide by $(\log B)^q$ and consider
    \begin{equation*}
        V(z_1,\dots,z_4;B):=\frac{B(\frac 1 2 (\log((B/T_2^d)z_3z_4))-\log(z_1z_2))^q}{(\log B)^q z_1z_2z_3z_4} \cdot V'(z_1,\dots,z_4;B),
    \end{equation*}
    where $V'$ is the indicator function of the set of all $z_1,\dots,z_4 \ge 1$ satisfying
    \begin{equation}\label{eq:indicator_function_i}
        z_1^2z_2^2z_3^{-1}z_4^{-1},z_1^2z_3z_4,z_2^2z_3z_4 \le B/T_2^d.
    \end{equation}
    Clearly $V$ vanishes unless $z_1,\dots,z_4 \ll B$, and $V \ll B/(z_1\cdots z_4)$, hence we may apply \cite[Proposition~7.2]{DF14} inductively to obtain (after multiplying by $(\log B)^q$ again) that
    \begin{align*}
        &\sums{\afrb' \in \IK^4\\\eqref{eq:condition_T2_i}} \frac{|\mu_K|\theta_0(\afrb')\theta(\afrb')B(\frac 1 2 (\log((B/T_2^d)\N(\afr_3\afr_4)))-\log(\N(\afr_1\afr_2)))^q}{R_K q! \N(\afr_1\cdots\afr_4)}\\
        ={}& \frac{|\mu_K|\rho_K^4\theta_1}{R_K}V_0(B)+O(B(\log B)^{3+q}(\log \log B)),
    \end{align*}
    where
    \begin{equation*}
        V_0(B) = \frac{1}{q!}\ints{z_1,\dots,z_4 \ge 1\\\eqref{eq:indicator_function_i}} \frac{B(\frac 1 2 (\log((B/T_2^d)z_3z_4))-\log(z_1z_2))^q}{z_1z_2z_3z_4} \ddd z_1 \cdots \ddd z_4.
    \end{equation*}
    We substitute $z_i=(B/T_2^d)^{t_i}$, with $\log z_i=t_i \log(B/T_2^d)$. Hence
    \begin{equation*}
        \tfrac 1 2 (\log((B/T_2^d)z_3z_4))-\log(z_1z_2) = \log(B/T_2^d)\cdot(\tfrac 1 2(1+t_3+t_4)-t_1-t_2).
    \end{equation*}
    Furthermore, this substitution turns the range of integration into the polytope $P$. Therefore, we get
    \begin{equation*}
        V_0(B) = V_1\cdot B(\log(B/T_2^d))^{4+q} = V_1\cdot B(\log B)^{4+q}+O\left(\frac{B(\log B)^{4+q}}{\log \log B}\right).\qedhere
    \end{equation*}
\end{proof}

Next we wish to identify the constant $V_1$ in the previous lemma with a suitable rational multiple of the constant $\alpha(X)$ from our main theorem. The rational multiple corresponds to the fact that when introducing the symmetry conditions, we have restricted the $a_i$ to a smaller polytope, three copies of which partition the polytope in the original definition of $\alpha(X)$.

\begin{lemma}\label{lem:compare_V1_alpha_i}
    We have
    \begin{equation*}
        V_1 = \frac{2}{3} \alpha(X).
    \end{equation*}
\end{lemma}

\begin{proof}
    Our symmetries $s_1,s_2$ from Section~\ref{sec:symmetry_i} induce linear maps
    \begin{align*}
        f_1 : \RR^4 \to \RR^4,\quad &(t_1,t_2,t_3,t_4) \mapsto \left(t_1,\tfrac{1-t_3-t_4}{2},\tfrac{1-2t_2+t_3-t_4}{2},\tfrac{1-2t_2-t_3+t_4}{2}\right),\\
        f_2 : \RR^4 \to \RR^4,\quad &(t_1,t_2,t_3,t_4) \mapsto \left(\tfrac{1-t_3-t_4}{2},t_2,\tfrac{1-2t_1+t_3-t_4}{2},\tfrac{1-2t_1-t_3+t_4}{2}\right)
    \end{align*}
    of determinant $-1$ (since $f_1$ is the reflection on the hyperplane $2t_2+t_3+t_4=1$, and $f_2$ is the reflection on $2t_1+t_3+t_4=1$).
    
    We observe that $f_1,f_2$ leave the range of integration $P_\alpha$ of $\alpha(X)$ invariant. The range of integration $P$ of $V_1$ is the subset of $P_\alpha$ cut out by $2x_1+x_3+x_4 \le 1$ and $2x_2+x_3+x_4\le 1$, and we compute that $f_1(P)$ is the subset of $P_\alpha$ defined by $x_1 \le x_2$ and $2x_2+x_3+x_4 \ge 1$, while $f_2(P)$ is the subset of $P_\alpha$ defined by $x_2 \le x_1$ and $2x_1+x_3+x_4 \ge 1$. Therefore, $P_\alpha$ is the union of $P,f_1(P),f_2(P)$, whose interiors are disjoint.

    Since also $1-2t_1-2t_2+t_3+t_4$ in the integrand of $V_1$ and of $\alpha(X)$ is invariant under $f_1,f_2$, the claim follows (taking the factor $2$ appearing only in the definition of $\alpha(X)$ into account).
\end{proof}

Finally, we can deduce our main result.

\begin{proof}[Proof of the Theorem]
    We combine Proposition~\ref{prop:ideals_i} with Lemmas~\ref{lem:summation_over_afr_i} and \ref{lem:compare_V1_alpha_i}. By Lemma~\ref{lem:densities_under_symmetry_i}, the summation over $s \in S$ from Lemma~\ref{lem:symmetry_i} gives a factor of $6$. In total, $N_{\Ufr,V,H}(B)$ is
    \begin{align*}
         ={}& \frac{1}{|\mu_K|}\cdot 6\cdot \frac {R_K}{2^2|\Delta_K|} \cdot \rho_K^4\frac{|\mu_K|}{R_K}\theta_1 \cdot V_1 \cdot \left(\prod_{v \mid \infty}\omega_{H,v}(X)+O(T_2^{-d})\right) \cdot B(\log B)^{4+q}\\
         &+O\left(\frac{B(\log B)^{4+q}}{\log \log B}\right)\\
        ={}& \alpha(X)\frac{\rho_K^4}{|\Delta_K|} \prod_v\omega_{H,v}(X) B(\log B)^{4+q}+O\left(\frac{B(\log B)^{4+q}}{\log \log B}\right).\qedhere
    \end{align*}
\end{proof}

\section{The expected asymptotic formula}\label{sec:expected}

We compute the expected asymptotic formula as in \cite[\S 2]{CLT10}, \cite[\S 2]{Wil24}, \cite[\S 6]{Santens23}, \cite[\S 6]{DW24}. Let $U = X \setminus A_{12}$ be the complement of our boundary divisor. We observe that we are in the situation of \cite[\S 2]{Wil24} since $X$ is a split Fano variety, the log anticanonical bundle $\omega_X(A_{12})^\vee$ is big, and $A_{12}$ has only one component over $\overline{K}$. Since $\omega_X(A_{12})^\vee$ is also nef, $(X,A_{12})$ is a weak del Pezzo pair as in \cite[\S 2]{DW24}. The abelian group $E(U) = \Os_X(U)^\times / K^\times$ appearing in \cite[\S 2]{Wil24} is trivial because all invertible regular functions on $U$ are constant since $A_{12}$ has negative self intersection number $-1$. 

With the notation of \cite[\S 2.5]{Wil24}, we expect that
\begin{equation*}
    N_{\Ufr,V,H}(B) = c_\infty c_\fin B(\log B)^{b-1} (1+o(1)),
\end{equation*}
where the factors
\begin{align*}
    c_\infty &= \frac{1}{|\Delta_K|} \alpha_{A_{12}}(X) \prod_{v \mid \infty} \tau_{A_{12},v}(A_{12}(K_v)),\\
    c_\fin &= \rho_K^{\rk \Pic U} \prod_{\pfr} \left(1-\frac{1}{\N\pfr}\right)^{\rk \Pic U} \tau_{(X,D),\pfr}(\Ufr(\Os_\pfr)),
\end{align*}
of the leading constant and $b$ in the exponent of $\log B$ will be computed in the remainder of this section. By \cite[Theorem~6.11, \S 6.3]{Santens23}, this agrees with the conjectures of Santens.

In our case, for any archimedean place $v$, the $K_v$-analytic Clemens complex $\Cs_v^{\an}(A_{12})$ consists of the two faces $(\emptyset,X) \prec ((12),A_{12})$. Hence the archimedean analytic Clemens complex $\Cs_\infty^{\an}(A_{12}) = \prod_{v \mid \infty} \Cs_v^{\an}(A_{12})$ has only one maximal face $\mathbf{A} = (A_v)_{v \mid \infty} = ((12),A_{12})_{v \mid \infty}$. We have $\rk \Pic U = \rk \Pic X - 1 = 4$ by \cite[(10)]{Wil24}. Hence
\begin{equation}\label{eq:log_B_exponent}
    b = b_{\mathbf{A}} = \rk \Pic U + \sum_{v \mid \infty} \#A_v = 4 + r_1 + r_2 = 5 + q.
\end{equation}

To compute the Tamagawa measure, we argue as in \cite[\S 6, Lemma~23, Lemma~25]{DW24} and work in Cox coordinates with the chart
    \begin{align*}
        f: V' = X \setminus \{a_1a_2a_3a_4a_{13}=0\} &\to \AAA^2_K \setminus \{(0,0),(1,1)\},\\
        (a_1:\dots:a_{34})&\mapsto \left(\frac{a_2a_{12}}{a_3a_{13}},\frac{a_2a_{23}}{a_1a_{13}}\right)
    \end{align*}
    and its inverse
    \begin{align*}
        g: \AAA^2_K \setminus \{(0,0),(1,1)\} &\to X,\\
        (x,y) &\mapsto (1:1:1:1:x:1:1-x:y:y-x:y-1).
    \end{align*}
These isomorphisms (inverse to each other) are obtained from the blow-up map $\pi : X \to \PP^2_K$  by removing $p_2,p_4$ and the line through $p_1,p_3$ from $\PP^2_K$, and the exceptional divisors $A_1,\dots,A_4$ and the strict transform $A_{13}$ of this line from $X$.
    
Since the height $H$ in our Theorem is induced by the log-anticanonical sections $\{\Pt : P \in \Ps\}$ by Lemma~\ref{lem:heights_i}, we must use the same sections in the following computation. For the local measure $\tau_{(X,D),v}$ from \cite[\S 2.4.3]{CLT10}, the norm appearing in
    \begin{equation*}
        \ddd f_* \tau_{(X,D),v} = \|(\ddd x \wedge \ddd y) \otimes 1_{A_{12}} \|_{\omega_X(D),v}^{-1} \ddd x \ddd y
    \end{equation*}
    in Cox coordinates is
    \begin{equation*}
        \frac{|a_1^2a_2^{-1}a_3^2a_4^{-1}a_{13}^3|}{|a_{12}|
        \max_{P \in \Ps} |\Pt(a_1,\dots,a_{34})|}
    \end{equation*}
    hence
    \begin{equation*}
        \ddd f_* \tau_{(X,D),v} = \frac{1}{|x|_v M_v(x,y)} \ddd x \ddd y,
    \end{equation*}
    where
    \begin{equation*}
        M_v(x,y):=
        \max_{P \in \Ps} |P(y,1,x)|_v.
    \end{equation*}
    
\begin{lemma}
    For every finite place $v=\pfr$ of $K$, we have
    \begin{align*}
        &f(\Ufr(\Os_\pfr)\cap V'(K_\pfr))\\
        &= \{(x,y) \in K_\pfr^2 \setminus \{(1,1)\} : |x|_\pfr \ge 1\} \cup \{(x,y) \in K_\pfr^2 \setminus \{(0,0)\} : |y|_\pfr\le |x|_\pfr < 1\}.
    \end{align*}
\end{lemma}

\begin{proof}
    Since $f : V' \to \AAA^2_K \setminus \{(0,0),(1,1)\}$ is an isomorphism over $K$, we have
    \begin{equation*}
        f(V'(K_\pfr)) = K_\pfr^2 \setminus \{(0,0),(1,1)\}.
    \end{equation*}
    We observe that the preimage of $(x,y)$ can be represented by integral coordinates $(a_1,\dots,a_{34})$ satisfying the coprimality conditions corresponding to \eqref{eq:coprimality_i} as follows: Let $(x_1,x_2,x_3) \in \Os_\pfr^3$ be a coprime tuple with $x=\frac{x_3}{x_2}$ and $y=\frac{x_1}{x_2}$. We can then define
    \begin{equation*}
        a_1=\gcd(x_2,x_3),\ a_2=\gcd(x_1,x_3),\ a_3=\gcd(x_1,x_2),\ a_4=\gcd(x_1-x_2,x_1-x_3)
    \end{equation*}
    (where we mean a generator of the respective ideal) as well as
    \begin{equation*}
        a_{12}=\frac{x_3}{a_1a_2}, a_{13}=\frac{x_2}{a_1a_3}, a_{23}=\frac{x_1}{a_2a_3}, a_{14}=\frac{x_2-x_3}{a_1a_4}, a_{24}=\frac{x_1-x_3}{a_1a_4}, a_{34}=\frac{x_1-x_2}{a_3a_4}
    \end{equation*}
    and observe that this is an integral tuple satisfying the Plücker equations and the coprimality conditions that is mapped to $(x,y)$.

    Next, we note that $(x,y) \in f(\Ufr(\Os_\pfr))$ if and only if $v_\pfr(a_{12})=0$. Upon recalling that $x_1=a_2a_3a_{23}$, $x_2=a_1a_3a_{13}$, $x_3=a_1a_2a_{12}$, this is equivalent to $v_\pfr(x_3) \le \max\{v_\pfr(x_1),v_\pfr(x_2)\}$. Indeed, if $v_\pfr(a_{12}) > 0$, then $v_\pfr(a_3)=v_\pfr(a_{23})=v_\pfr(a_{13})=0$, hence $v_\pfr(x_3)=v_\pfr(a_1a_2a_{12})$ is greater than $v_\pfr(x_1)=v_\pfr(a_2)$ and $v_\pfr(x_2)=v_\pfr(a_1)$. Conversely, if $v_\pfr(a_{12}) = 0$, then
    \begin{equation*}
        v_\pfr(x_3) \le \max\{v_\pfr(a_1),v_\pfr(a_2)\} \le \max\{v_\pfr(x_1),v_\pfr(x_2)\}
    \end{equation*}
    (since $a_1,a_2$ are coprime).

    Hence the integrality condition is equivalent to $|x_3|_\pfr \ge |x_1|_{\pfr}$ or $|x_3|_\pfr \ge |x_2|_{\pfr}$ and hence to $|x|_\pfr \ge |y|_\pfr$ or $|x|_\pfr \ge 1$.
\end{proof}

The following result is compatible with the observation that we have $X(\FF_q) = q^2+5q+1$, so that $q^2+4q$ points remain after removing the boundary divisor $\PP^1$ with $q+1$ points; see also \cite[Lemma~3.1]{WilschIMRN}.

\begin{lemma}
    The expected $\pfr$-adic density is
    \begin{equation*}
         \left(1-\frac{1}{\N\pfr}\right)^4\tau_{(X,D),\pfr}(\Ufr(\Os_\pfr)) = \left(1-\frac{1}{\N\pfr}\right)^4\left(1+\frac{4}{\N\pfr}\right) = \omega_{H,\pfr}(X).
    \end{equation*}
\end{lemma}

\begin{proof}
    As in \cite[Lemma~24]{DW24}, we must compute
    \begin{equation*}
        \tau_{(X,D),\pfr}(\Ufr(\Os_\pfr)) = \ints{x,y \in K_{\pfr}\\|x|_\pfr\ge 1\text{ or }|y|_\pfr\le|x|_\pfr<1} \frac{1}{|x|_\pfr M_\pfr (x,y)} \ddd x \ddd y.
    \end{equation*}
    By \eqref{eq:height_gcd_i}, we have
    \[|x|_\pfr M_\pfr(x,y)=|x|_\pfr \cdot \max\{1,|y|_\pfr\} \cdot \max\{|y-1|_\pfr,|x-1|_\pfr\}.\]
    We note that this coincides with the expression
    \[\frac{\max\{1,|x|_\pfr\} \cdot \max\{1,|y|_\pfr\} \cdot \max\{|x|_\pfr,|y|_\pfr\} \cdot \max\{|y-1|_\pfr,|x-1|_\pfr\}}{\max\{1,|x|_\pfr,|y|_\pfr\}}\]
    in the range of integration above. The computation in \cite[Lemma 2.2]{BD24} now reveals that the integral of the latter expression over all of $\mathbb{Q}_\pfr^2$ equals $1+\frac{5}{\N\pfr}+\frac{1}{\N\pfr^2}$. It thus suffices to prove that the complementary range where $|x|_{\pfr}<1$ and $|x|_\pfr<|y|_\pfr$ (corresponding to $D$) contributes $\frac{1}{\N\pfr}+\frac{1}{\N\pfr^2}$. Indeed, this integral equals
    \[\ints{|x|_\pfr<1\\|x|_\pfr<|y|_\pfr} \frac{1}{|y|_\pfr \cdot \max\{1,|y|_\pfr\}} \ddd x\ddd y=\frac{1}{\N\pfr}+\frac{1}{\N\pfr^2},\]
    where the case $|y|_\pfr \ge 1$ contributes $\frac{1}{\N\pfr}$ and the case $|y|_\pfr<1$ contributes $\frac{1}{\N\pfr^2}$.
\end{proof}

\begin{lemma}
    For $v\mid\infty$, the expected density is $\tau_{A_{12},v}(A_{12}(K_v)) = \omega_{H,v}(X)$.
\end{lemma}

\begin{proof}
    Let $v$ be an archimedean place.
    To compute the unnormalized Tamagawa volume of $A_{12}$, we integrate
    \begin{equation*}
        \|\ddd y\|_{\omega_{A_{12},v}}^{-1} = \lim_{x \to 0} \left(|x|\|(\ddd x \wedge \ddd y)\otimes 1_{A_{12}}\|_{\omega_X(A_{12}),v}^{-1}\right)
    \end{equation*}
    and obtain (with twice the Lebesgue measure in the complex case)
    \begin{align*}
        \tau_{A_{12},v}'(A_{12}(K_v)) &= \int_{K_v} \lim_{x \to 0} \frac{|x|_v}{|x|_v M_v(x,y)} \ddd y\\
        &=\int_{K_v} \frac{1}{M_v(0,y)} \ddd y.
    \end{align*}
    
    For real $v$, we use
    \begin{equation*}
        \int_{m|t|^2 \le 1} |t| \ddd t = 2 \int_0^{m^{1/2}} t \ddd t = \frac 1 m
    \end{equation*}
    (which holds for any positive $m \in \RR$) to see that our integral over $y$ is
    \begin{align*}
        \tau_{A_{12},v}'(A_{12}(\RR)) &= \int_{M_v(0,y)|t|^2 \le 1} |t| \ddd t \ddd y\\
        &=\int_{
        \max_{P \in \Ps} |P(t+z,t,0)| \le 1} \ddd t \ddd z\\
        &=\vol\{(t,z) \in \RR^2 : \max_{P \in \Ps} |P(t+z,t,0)| \le 1\}
    \end{align*}
    where we use the transformation $z=|t|y$ in the second step. Finally, we normalize this with the factor $c_\RR = 2$.
    
    For complex $v$, we have (with the usual Lebesgue measure)
    \begin{align*}
        \tau_{A_{12},v}'(A_{12}(\CC)) &= 2\int_{\CC} \frac{1}{M_v(0,y)} \ddd y\\
        &=4 \ints{M_v(0,y)t^2 \le 1\\t \in \RR_{>0},\ y \in \CC} t \ddd t \ddd y\displaybreak[4]\\
        &=\frac{4}{\pi}\int_{|u_1|_v^2M_v(0,u_2)\le 1} |u_1|_v \ddd u_1 \ddd u_2\\
        &=\frac{4}{\pi} \int_{\max_{P \in \Ps} |P(y_1+y_2,y_1,0)| \le 1} \ddd y_1 \ddd y_2\\
         &=\frac{4}{\pi}\vol\{(y_1,y_2) \in \CC^2 : \max_{P \in \Ps} |P(y_1+y_2,y_1,0)| \le 1\},
    \end{align*}
    using polar coordinates, and where $u_1,u_2,y_1,y_2$ are complex coordinates (with $y_1=u_1,y_2=u_1u_2$). We normalize this with the factor $c_\CC = 2\pi$ \cite[\S 2.6.1]{Wil24}.
\end{proof}

\begin{lemma}
    We have $\alpha_{A_{12}}(X) = \alpha(X)$.
\end{lemma}

\begin{proof}
    See \cite[\S 2.2]{Wil24} for the definition of $\alpha_{A_{12}}(X)$. By \cite[Theorem~2.4.1]{Wil24} and \eqref{eq:log_B_exponent}, the $\Ab$-divisor class group $\Pic(U;\Ab)$ has rank $b_\Ab' = b_\Ab = 5+q$ and is the quotient of $\Pic(X)\oplus \ZZ^{\Omega_\infty}$ with respect to the relation $([A_{12}],(0)_{v \mid \infty}) \sim (0,(1)_{v \mid \infty})$, where $[A_{12}] = \ell_0-\ell_1-\ell_2$ is the class of $A_{12}$ in $\Pic(X)$.
    
    Hence the image of the log-anticanonical divisor class $-K_X-[A_{12}] = 2\ell_0-\ell_3-\ell_4$ is $(2\ell_1+2\ell_2-\ell_3-\ell_4,(2)_{v \mid \infty})$, and the effective cone $\Eff_\Ab$ is generated by $(\ell_i,(0)_{v \mid \infty})$ (corresponding to $A_i$), $(\ell_1+\ell_2-\ell_i-\ell_j,(1)_{v \mid \infty})$ (corresponding to $A_{ij}$), and the $r_1+r_2$ elements $(0,(0,\dots,1,\dots,0))$.

    By \cite[Remark~2.2.9(iv)]{Wil24}, $\alpha_{A_{12}}(X)$ is
    \begin{equation*}
         b_\Ab' \vol\bigwhere{((t_1,\dots,t_4),(t_v)_{v \mid \infty}) \in \RR_{\ge 0}^4 \times \RR_{\ge 0}^{\Omega_\infty}}{&t_1+t_2-t_i-t_j+\sum_{v \mid \infty} t_v \ge 0\\&2t_1+2t_2-t_3-t_4+2\sum_{v \mid \infty} t_v \le 1}.
    \end{equation*}
    We make the change of coordinates $t=2t_1+2t_2-t_3-t_4+2\sum_{v \mid \infty} t_v$, eliminating $t_w$ for a fixed $w\mid\infty$, with Jacobian $2$. Therefore, $\alpha_{A_{12}}(X)= \frac{b_\Ab'}{2}\vol(P)$ where
    \begin{equation*}
        P = \bigwhere{((t_1,\dots,t_4),(t_v)_{v \ne w},t) \in \RR_{\ge 0}^{4+q}\times \RR}{&t \le 1, \ 2t_i+2t_j-t_3-t_4\le t\\& \sum_{v \ne w} t_v \le \frac 1 2 (t-2t_1-2t_2+t_3+t_4)}.
    \end{equation*}
    
    We observe that $P$ is a pyramid of height $1$ whose apex is the origin and whose base is $P_0 \times \{1\}$ in the hyperplane $\{t=1\}$ in $\RR^{5+q}$, where
    \begin{equation*}
        P_0=\bigwhere{((t_1,\dots,t_4),(t_v)_{v \ne w}) \in \RR_{\ge 0}^{4+q}}{&2t_i+2t_j-t_3-t_4\le 1\\& \sum_{v \ne w} t_v \le \frac 1 2 (1-2t_1-2t_2+t_3+t_4)}.
    \end{equation*}
    Using $b_\Ab'=5+q$, the volume of this pyramid is
    \begin{equation*}
        \frac{1}{b_\Ab'}\vol(P_0) = \frac{1}{b_\Ab'}\ints{t_1,\dots,t_4 \ge 0\\2t_i+2t_j-t_3-t_4 \le 1\\} \left(\ints{t_v \ge 0\\ \sum_{v \ne w} t_v \le \frac 1 2 (1-2t_1-2t_2+t_3+t_4)} \ddd t_v\right) \ddd t_1\cdots \ddd t_4.
    \end{equation*}
    Here, the inner integral is
    \begin{align*}
        \left(\frac 1 2 (1-2t_1-2t_2+t_3+t_4)\right)^q \vol\Bigg\{(t_v) \in \RR_{\ge 0}^q : \sum_{v \ne w} t_v \le 1\Bigg\},
    \end{align*}
    and the volume of this simplex is $1/q!$.
\end{proof}

\bibliographystyle{amsalpha}

\bibliography{dp5_integral}

\end{document}